\documentclass[12pt, leqno]{amsart}

\usepackage{
amsmath,
amssymb,
amsthm,
mathrsfs,
amsfonts,
ytableau,
enumitem,
comment,
upgreek,
soul}
\usepackage{tikz,tikz-cd}
\usetikzlibrary{decorations.pathreplacing,decorations.pathmorphing}
\usetikzlibrary{arrows,arrows.meta}
\usetikzlibrary{shapes,shapes.geometric,shapes.symbols}
\usetikzlibrary{matrix,calc}
\usetikzlibrary{trees,positioning,chains}
\usepackage[poly,all]{xy}
\usepackage[colorinlistoftodos, textwidth = 2.3cm]{todonotes}
\usepackage{chngcntr}

\ytableausetup{mathmode, boxsize=1.2em}

\newtheorem{theorem}{Theorem}[section]
\newtheorem{proposition}[theorem]{Proposition}
\newtheorem{lemma}[theorem]{Lemma}
\newtheorem{corollary}[theorem]{Corollary}

\newtheorem*{claim*}{Claim}

\theoremstyle{definition}
\newtheorem{algorithm}[theorem]{Algorithm}
\newtheorem{example}[theorem]{Example}
\newtheorem{definition}[theorem]{Definition}
\newtheorem{remark}[theorem]{Remark}

\numberwithin{equation}{section} \numberwithin{figure}{section}
\numberwithin{table}{section}

\def\Z{\mathbb Z}
\def\C{\mathbb C}

\newcommand{\nc}{\newcommand}

\nc{\SG}{\mathfrak{S}}
\nc{\PCT}{\mathrm{PCT}}
\nc{\SPCT}{\mathrm{SPCT}}
\nc{\RT}{\mathrm{RT}}
\nc{\SRT}{\mathrm{SRT}}
\nc{\RCT}{\mathrm{RCT}}
\nc{\SRCT}{\mathrm{SRCT}}
\nc{\stan}{\mathrm{stan}}
\nc{\Span}{\mathrm{span}}
\nc{\comp}{\mathrm{comp}}
\nc{\rmst}{\mathrm{st}}
\nc{\Des}{\mathrm{Des}}
\nc{\set}{\mathrm{set}}
\nc{\wt}{\mathrm{wt}}
\nc{\ch}{\mathrm{ch}}
\nc{\id}{\mathrm{id}}
\nc{\lex}{\mathrm{lex}}
\nc{\Qsym}{\mathrm{QSym}}
\nc{\sh}{\mathrm{sh}}
\nc{\bfS}{\mathbf{S}}
\nc{\bfF}{\mathbf{F}}
\nc{\calS}{\mathcal{S}}
\nc{\alphamax}{\alpha_{\rm max}}
\nc{\brho}{\overline{\rho}}
\nc{\bphi}{\overline{\phi}}
\nc{\calV}{\mathcal{V}}
\nc{\calR}{\mathcal{R}}
\nc{\calG}{\mathcal{G}}
\nc{\ineq}{\preceq}
\nc{\ineqta}{\ineq^t_\alpha}
\nc{\tal}{\widetilde{\alpha}}
\nc{\tbe}{\widetilde{\beta}}
\nc{\tS}{\widetilde{\bfS}}
\nc{\oT}{\overline{T}}
\nc{\oU}{\overline{U}}
\nc{\hpi}{\pi}
\nc{\opi}{\overline{\pi}}
\nc{\calP}{\mathcal{P}}
\nc{\calC}{\mathcal{C}}
\nc{\rmtop}{\mathrm{top}}
\nc{\rad}{\mathrm{rad}}
\nc{\bfP}{\mathbf{P}}
\nc{\SET}{\mathrm{SET}}
\nc{\rev}{\mathrm{r}}
\nc{\Th}{\theta}
\nc{\Thact}{*}
\nc{\ThbfP}{\Th[\bfP_\alpha]}
\nc{\hThbfP}{\ThbfP}
\nc{\hThbfPt}{\widehat{\Th[\bfP_{\alpha^t}]}}
\nc{\hPhi}{\widehat{\Phi}}
\nc{\hPsi}{\widehat{\Psi}}
\nc{\htau}{\widehat{\tau}}
\nc{\otau}{\overline{\tau}}
\nc{\mpsi}{\psi}
\nc{\tab}{\tau}
\nc{\lew}{\le_{\mathrm{w}}}
\nc{\tcd}{\mathtt{cd}}
\nc{\trd}{\mathtt{rd}}
\nc{\rmr}{\mathrm{r}}
\nc{\rmc}{\mathrm{c}}
\nc{\rmt}{\mathrm{t}}
\nc{\TOP}{\mathrm{top}}
\nc{\SOC}{\mathrm{soc}}
\nc{\bubact}{\,\scalebox{0.6}{$\bullet$}\,}
\nc{\col}{\rm col}
\nc{\row}{\rm row}
\nc{\calE}{\mathcal{E}}
\nc{\calEsa}{\mathcal{E}^\sigma_\alpha}
\nc{\bsig}{\boldsymbol{\sigma}}
\nc{\tauC}{\tau_{\scalebox{0.5}{$C$}}}
\nc{\htauC}{\widehat{\tau}_{\scalebox{0.5}{$C$}}}
\nc{\kappaC}{\kappa_{\scalebox{0.5}{$C$}}}
\nc{\hkappaC}{\widehat{\kappa}_{\scalebox{0.5}{$C$}}}
\nc{\mfPa}{\mathfrak{P}_\alpha}
\nc{\rma}{\mathrm{a}}
\nc{\rmd}{\mathrm{d}}

\nc{\SPCTsa}{\SPCT^\sigma(\alpha)}
\nc{\bfSsa}{\bfS_\alpha^\sigma}
\nc{\hPhisa}{\widehat{\Phi}^\sigma_\alpha}
\nc{\bfSsaC}{{\bfS}^\sigma_{\alpha,C}}

\nc{\ra}{\rightarrow}

\tikzset{
>=stealth',
punktchain/.style={
    rectangle, 
    rounded corners, 
    draw=black, very thick,
    text width=11em, 
    minimum height=3em, 
    text centered, 
    on chain,
    scale=0.8},
small punktchain/.style={
    rectangle, 
    rounded corners, 
    draw=black, very thick,
    text width=8em, 
    minimum height=3em, 
    text centered, 
    on chain,
    scale=0.8},
line/.style={draw, thick, <-},
  element/.style={
    tape,
    top color=white,
    bottom color=blue!50!black!60!,
    minimum width=8em,
    draw=blue!40!black!90, very thick,
    text width=7em, 
    minimum height=4em, 
    text centered, 
    on chain},
  every join/.style={<-, thick,shorten >=1pt},
  decoration={brace},
  tuborg/.style={decorate},
  tubnode/.style={midway, right=2pt},
}

\nc{\yh}[1]{\todo[size=\tiny,color=blue!10]{#1 \\ \hfill --- Young-Hun}}
\nc{\YH}[1]{\todo[size=\tiny,inline,color=blue!10]{#1
		\\ \hfill --- Young-Hun}}
\nc{\sy}[1]{\todo[size=\tiny,color=magenta!10]{#1 \\ \hfill --- Sun-Young}}
\nc{\SY}[1]{\todo[size=\tiny,inline,color=magenta!10]{#1
		\\ \hfill --- Sun-Young}}
\nc{\nt}[1]{\todo[size=\tiny,color=red!40]{#1 \\ \hfill --- Note}}
\nc{\NT}[1]{\todo[size=\tiny,inline,color=red!40]{#1
		\\ \hfill --- Note}}

\newenvironment{red}{\relax\color{red}}{\relax}
\newenvironment{blue}{\relax\color{blue}}{\hspace*{.5ex}\relax}

\newenvironment{magenta}{\relax\color{magenta}}{\hspace*{.5ex}\relax}

\nc{\ber}{\begin{red}}
\nc{\er}{\end{red}}
\nc{\beb}{\begin{blue}}
\nc{\eb}{\end{blue}}
\nc{\bema}{\begin{magenta}}
\nc{\ema}{\end{magenta}}

\title[Modules of the $0$-Hecke algebra]{Modules of the $0$-Hecke algebra arising from standard permuted composition tableaux}

\author[S.-I. Choi]{Seung-Il Choi}
\address{Research Institute of Mathematics, Seoul National University, Seoul 08826, Republic of Korea}
\email{ignatioschoi@snu.ac.kr}

\author[Y.-H. Kim]{Young-Hun Kim}
\address{Department of Mathematics, Sogang University, Seoul 04107, Republic of Korea \& Research Institute for Basic Science, Sogang University, Seoul 04107, Republic of Korea}
\email{yhkim14@sogang.ac.kr}

\author[S.-Y. Nam]{Sun-Young Nam}
\address{Department of Mathematics, Sogang University, Seoul 04107, Republic of Korea}
\email{synam.math@gmail.com}

\author[Y.-T. Oh]{Young-Tak Oh}
\address{Department of Mathematics, Sogang University, Seoul 04107, Republic of Korea \& Korea Institute for Advanced Study, Seoul 02455, Republic of Korea}
\email{ytoh@sogang.ac.kr}
\thanks{The first author was supported by the National Research Foundation of Korea(NRF) grant funded by the Korea government(MSIT) 
(No. NRF-2019R1C1C1010668).}
\thanks{The second and fourth authors were supported by NRF grant funded by the Korean Government (No. NRF-2018R1D1A1B07051048).}
\thanks{
The third author was supported by Basic Science Research Program through NRF funded by the Ministry of Education (No. NRF-2019R1I1A1A01062658).}


\date{\today}

\begin{document}

\maketitle
	
\begin{abstract}
We study the $H_n(0)$-module $\bfS^\sigma_\alpha$ due to Tewari and van Willigenburg, 
which was constructed using new combinatorial objects called standard permuted composition tableaux
and decomposed into cyclic submodules.
First, we show that every direct summand appearing in their decomposition is indecomposable
and characterize when $\bfS^\sigma_\alpha$ is indecomposable. 
Second, we find characteristic relations among $\bfSsa$'s and expand the image of $\bfSsa$ under the quasi characteristic in terms of quasisymmetric Schur functions. 
Finally, we show that the canonical submodule of $\bfSsa$ appears as a homomorphic image of a projective indecomposable module.
\end{abstract}

\section{Introduction}
Let $(W,S)$ be a finite Coxeter system and $H_W(q)$ the corresponding Hecke algebra defined over $\mathbb C$ with indeterminate $q$.
The $0$-Hecke algebra $H_W(0)$ is the degenerate form of $H_W(q)$ obtained by the specialization $q=0$. In case where $W$ is the symmetric group $\SG_n$ and $S$ the set of simple transpositions, we simply write it as $H_n(0)$.   
It is well known that $H_n(0)$ is of tame representation type for $n=3$ and 
of wild representation type for $n\ge 4$ (see \cite{11BG,02DHT}).

It was Norton~\cite{79Norton} who completely classified all projective indecomposable $H_n(0)$-modules (for short, PIMs) and simple $H_n(0)$-modules, up to equivalence. These modules correspond in a natural way with subsets of $[n-1]$, equivalently with compositions of $n$. 
Letting $\calP_I$ be the PIM associated to each $I \subseteq [n-1]$ and $\rad \; \calP_I$ the radical of $\calP_I$, the top of $\calP_I$, that is, $\rmtop(\calP_I) := \calP_I / \rad \; \calP_I$ is simple and every simple $H_n(0)$-module appears in this way.

The representation theory of $0$-Hecke algebras (of type A) has strong connections to quasisymmetric functions.
The most striking feature might be that the Grothendieck ring is isomorphic to
$\Qsym$, the ring of quasisymmetric functions,
which is a quasi-analogue of the Frobenius correspondence between characters of symmetric groups and symmetric functions.
The isomorphism is called the {\em quasisymmetric characteristic} under which $\rmtop(\calP_I)$ is mapped to $F_{\comp(I)}$, where $\comp(I)$ is the composition corresponding to $I$ and $F_{\comp(I)}$ is the fundamental quasisymmetric function attached to $\comp(I)$ (see \cite{96DKLT}).

Suppose that we have a noteworthy basis for $\Qsym$.
From the viewpoint of the Frobenius correspondence, it would be quite natural to ask whether basis elements appear as the image of module
isomorphism classes under the quasisymmetric characteristic. 
This work has been done for quasisymmetric Schur functions~\cite{15TW}, quasisymmetric functions forming the dual immaculate basis~\cite{15BBSSZ}, and extended Schur functions~\cite{19Searles}. 

Quasisymmetric Schur functions were introduced in~\cite{11HLMW} as an analogue of Schur functions. 
Given a composition $\alpha$ of $n$, denoted by $\alpha \models n$, the quasisymmetric Schur function $\calS_\alpha$ is defined by
$\sum_\gamma \mathcal{A}_\gamma$,
where the sum is over weak compositions $\gamma$ corresponding to $\alpha$
and $A_\gamma$ is the Demazure atom.
Later, Tewari and van Willigenburg~\cite{15TW} showed that it can be written as
\begin{equation*}
\sum_{\tau \in \SRCT(\alpha)} F_{\comp(\tau)},
\end{equation*}
where $\SRCT(\alpha)$ is the set of standard reverse composition tableaux of shape $\alpha$
and $\comp(\tau)$ is the composition of $n$ corresponding to the set of all descents of $\tau$. 
They form a basis for $\Qsym$ 
and, as inferred from the name, they refine many properties of Schur functions 
such as Pieri rule, Kostka numbers, and Littlewood--Richardson rule.
For quasi-versions for these themes, refer to~\cite{11BLW, 11HLMW, 11HLMW2, 13LMvW}. 

Considering the fact that each $\rmtop(\calP_I)$ is 1-dimensional and its characteristic image is $F_{\comp(I)}$, one can naively expect that the $\mathbb C$-span of $\SRCT(\alpha)$ is a feasible candidate
for a module whose characteristic image is given by $\calS_\alpha$. 
This expectation turns out to be affirmative by Tewari and van Willigenburg~\cite{15TW} who successfully defined an $H_n(0)$-action on $\SRCT(\alpha)$ and showed that the resulting $H_n(0)$-module, denoted by  $\bfS_\alpha$, has the desired property. 
Further, they revealed numerous interesting properties concerning these modules.

One can give a colored graph structure on $\SRCT(\alpha)$ by letting 
$$\tau_1 \overset{i}{\to} \tau_2 \quad \text{if and only if} \quad \tau_2=\pi_i \cdot \tau_1.$$
Here $\pi_i$ is a generator of $H_n(0)$ satisfying~\eqref{eq: Hecke gen}.
Every connected component has a unique source tableau and a unique sink tableau.
Let $\calE(\alpha)$ be the set of connected components
and $\bfS_{\alpha,E}$ be the $H_n(0)$-submodule of $\bfS_\alpha$ whose underlying space is the $\C$-span of $E$ for $E \in \calE(\alpha)$.
As $H_n(0)$-modules, we have
\begin{align*}
\bfS_\alpha \cong \bigoplus_{E \in \mathcal{E}(\alpha)} \bfS_{\alpha,E}.
\end{align*}
It was shown in \cite{15TW} that the component 
$$
E_\alpha :=\{\tau\in \SRCT(\alpha) \mid \text{entries in each column of $\tau$ increase from top to bottom}\}
$$
is equipped with notable properties such as 
$\bfS_{\alpha,E_\alpha}$ is indecomposable and $\{\ch([\bfS_{\alpha,E_\alpha}]) \mid \alpha \models n\}$ forms a basis for the homogeneous component $\Qsym_n$ of degree $n$ in $\Qsym$. 
Here $\ch$ denotes the quasisymmetric characteristic.
The indecomposability problem has been completely settled by K\"onig by showing that $\bfS_{\alpha,E}$ is indecomposable for every $E \in \mathcal{E}(\alpha)$ (see~\cite[Theorem 4.11]{19Konig}).

Recently, Tewari and van Willigenburg \cite{19TW} introduced new 
combinatorial objects, called standard permuted composition tableaux, by weakening the condition on the first column in standard reverse composition tableaux.
Given $\alpha \models n$ and $\sigma \in \SG_{\ell(\alpha)}$, a \emph{standard permuted composition tableau} of shape $\alpha$ and type $\sigma$ differs from a standard reverse composition tableau in that 
the standardization of the word obtained by reading the first column from top to bottom is $\sigma$.
Here $\ell(\alpha)$ is the length of $\alpha$.

In the following, let $\alpha \models n$ and $\sigma \in \SG_{\ell(\alpha)}$. 
Let $\SPCT^\sigma(\alpha)$ be the set of standard permuted composition tableaux of shape $\alpha$ and type $\sigma$. 
It was shown in~\cite{19TW} that $\SPCT^\sigma(\alpha)$ has an $H_n(0)$-action and the resulting $H_n(0)$-module, denoted by  $\bfS_\alpha^\sigma$, shares certain properties with $\bfS_\alpha$. 
For instance, in the same way as in \cite{15TW}, one can give a colored graph structure on $\SPCT^\sigma(\alpha)$.
Let $\calE^\sigma(\alpha)$ be the set of connected components and 
$\bfS^\sigma_{\alpha,E}$ be the $H_n(0)$-submodule of $\bfS^\sigma_\alpha$ whose underlying space is the $\C$-span of $E$ for $E \in \calE^\sigma(\alpha)$.
As $H_n(0)$-modules, we have
\begin{align*}
\bfS_\alpha^{\sigma} \cong \bigoplus_{E \in \mathcal{E}^\sigma(\alpha)} \bfS_{\alpha,E}^{\sigma}.
\end{align*}
Compared to $\bfS_\alpha$, the structure of $\bfS_\alpha^{\sigma}$ has not been understood well up to now,
which will be the main concern of the present paper. 
We here demonstrate some noteworthy properties. 

The first objective of this paper concerns 
the indecomposability of $\bfS^{\sigma}_{\alpha,E}$ and of $\bfS^{\sigma}_{\alpha}$.
We show that under suitable adjustments, K\"onig's method~\cite{19Konig} still works in our case.
This leads us to conclude that 
for every $E \in \mathcal{E}^{\sigma}(\alpha)$, the $H_n(0)$-module $\bfS^{\sigma}_{\alpha,E}$ is indecomposable
(Theorem~\ref{thm: eq class indecomp}).
To avoid excessive overlap with K\"onig's proof,
we present the roadmap to the proof of~\cite[Theorem 4.11]{19Konig} and our adjustments for the proof of Theorem~\ref{thm: eq class indecomp} instead of giving the full verification in Appendix~\ref{sec: indecomp}.

We say that $\bfS^{\sigma}_{\alpha}$ is an {\em SPCT-cyclic} $H_n(0)$-module if it is cyclically generated by a single standard permuted composition tableau.
Thanks to Theorem \ref{thm: eq class indecomp}, $\bfS^{\sigma}_{\alpha}$ is SPCT-cyclic if and only if it is  indecomposable.
The former condition is equivalent to saying that $\SPCT^\sigma(\alpha)$ contains exactly one source tableau since every connected component of $\SPCT^\sigma(\alpha)$ has a unique source.
Whenever $\alpha$ is \emph{$\sigma$-compatible}, then $\SPCT^\sigma(\alpha)$ is not empty (Proposition~\ref{Prop: comparibility}) and
there always exists one particular source tableau in $\SPCT^\sigma(\alpha)$, which is referred as \emph{the canonical source tableau} in this paper. 
Indeed it is an SPCT-analogue of the source tableau in $E_\alpha$.
We characterize when $\bfS^{\sigma}_{\alpha}$ is indecomposable by using the notion of \emph{$\sigma$-simplicity}.
To do this, we see that $\alpha$ is $\sigma$-simple if and only if the canonical source tableau is a unique source tableau in $\SPCT^\sigma(\alpha)$.
In the special case where $\sigma$ is the longest element $w_0$ in $\SG_{\ell(\alpha)}$,
the characterization says that
$\bfS^{w_0}_{\alpha}$ is a nonzero SPCT-cyclic $H_n(0)$-module if and only if $\alpha$ is a partition of $n$ (Theorem \ref{thm: SPCT-cyclic} and Corollary \ref{Coro318}).

The second objective is to find characteristic relations among $\bfSsa$'s and, using those, expand $\ch([\bfSsa])$ in terms of quasisymmetric Schur functions. 
Let us briefly explain our approach.    
Using the maps $\brho_{\sigma}$ and $\bphi_{\sigma}$ given in \cite{19TW},
we define a bijection $\mpsi^{\sigma_2}_{\sigma_1}$ between 
$\coprod_{\tal = \lambda} \SPCT^{\sigma_1}(\alpha)$ and $\coprod_{\tal = \lambda} \SPCT^{\sigma_2}(\alpha)$ for a partition $\lambda$ and $\sigma_1, \sigma_2 \in \SG_{\ell(\lambda)}$.
Then we show that if there is $1 \le i \le n-1$ such that $\ell(\sigma s_i) < \ell(\sigma)$, then 
\[
\mpsi^{\sigma s_i}_{\sigma}(\SPCTsa) = \coprod_{\alpha = \beta \bubact \pi_{i}} \SPCT^{\sigma s_i}(\beta),
\]
which can also be rephrased as 
\begin{align*}
\ch([\bfSsa]) = \sum_{\alpha = \beta \bubact \hpi_{i}} \ch([\bfS^{\sigma s_i}_\beta]).
\end{align*}
Here $\bubact$ is an $H_n(0)$-action defined on the set of compositions of length $n$ as in~\eqref{eq: bubble sorting operator}.  
Applying this relation repeatedly leads to the multiplicity-free expansion
\begin{align*}
\ch([\bfSsa]) &= \sum_{\alpha = \beta \bubact \hpi_{\sigma}} \calS_{\beta}.
\end{align*}
If $\alpha$ is a partition and $\sigma=w_0$,
then we have $\ch([\bfS_\alpha^{w_0}]) = s_\alpha$,
where $s_\alpha$ is the Schur function.
As a byproduct of this expansion, we construct a new $\Z$-basis of $\Qsym_n$ given by
\[
\{ \ch([\bfS_\lambda^{\sigma}]) \mid \lambda \vdash n \text{ and } \sigma \in \SG_{\ell(\lambda)}/ (\SG_{\ell(\lambda)})_{\lambda} \},
\]
where $(\SG_{\ell(\lambda)})_{\lambda} = \{\sigma \in \SG_{\ell(\lambda)} \mid \lambda \cdot \sigma = \lambda\}$ (Theorem \ref{thm: charcteristic of SPCT}, Corollary \ref{cor: qschur exp}, and Corollary~\ref{cor: basis for Qsym}). 

To state the final objective, 
let us consider the connected component $C \in \mathcal{E}^\sigma(\alpha)$ containing the canonical source tableau.
We call ${\bfS}^\sigma_{\alpha,C}$ the \emph{canonical submodule of $\bfSsa$}.
The third objective, which is most interesting from the representation-theoretical viewpoint,     
is to find the projective cover of ${\bfS}^\sigma_{\alpha,C}$.
Indeed it turns out to be a PIM.
To achieve our purpose, we construct a surjective $H_n(0)$-module homomorphism   
from $\calP_{\set(\alpha\cdot\sigma^{-1})}$
to $\bfS^{\sigma}_{\alpha,C}$, where $\set(\alpha\cdot\sigma^{-1})$ is the subset of $[n-1]$ corresponding to $\alpha\cdot\sigma^{-1}$ (Theorem \ref{Thm:main Sec5}).
This map is more or less simply defined and the kernel can be explicitly written.
In the special case where $\alpha = (\alpha_1, \ldots, \alpha_\ell)$ is a partition and $\sigma=w_0$, 
the composition $\alpha\cdot w_0^{-1}$ is nothing but $\alpha^\rev := (\alpha_\ell, \ldots, \alpha_1)$, thus  
$\bfS^{w_0}_\alpha$ is a homomorphic image of $\calP_{\set(\alpha^\rev)}$.
As a significant consequence of Theorem \ref{Thm:main Sec5}, we derive that   
$\bfS^{\sigma}_{\alpha,C}$ is projective if and only if 
either $\alpha = (1,1,\ldots, 1)$ or 
$\alpha = (k,1,1,\ldots,1)$ for some $k \ge 2$ and $\sigma(1)=\ell(\alpha)$.
In this case, $\alpha$ is $\sigma$-simple and $\bfS^{\sigma}_{\alpha}$ is isomorphic to $\calP_{\set(\alpha^\rmr)}$ (Corollary \ref{projective cover is projective}).

This paper is organized as follows.
In Section~\ref{Preliminaries}, we introduce the prerequisites 
on the $0$-Hecke algebra, the quasisymmetric characteristic, standard permuted composition tableaux
and an $H_n(0)$-module arising from them.
In Section~\ref{sec: Indecomposability}, we deal with 
the indecomposability of $\bfS^{\sigma}_{\alpha,E}$ and of $\bfS^{\sigma}_{\alpha}$.
The roadmap to the proof of the indecomposableness of $\bfS^{\sigma}_{\alpha,E}$ is contained in~Appendix~\ref{sec: indecomp}.
In Section~\ref{sec: Chac ral and quasi Schur exp},
we provide characteristic relations among $\bfSsa$'s and the expansion of $\ch([\bfSsa])$ in terms of quasisymmetric Schur functions.
Section~\ref{Sect5} is devoted to the study of the projective cover of the canonical submodule of $\bfSsa$.

\section{Preliminaries}\label{Preliminaries}

In this section, $n$ denotes a nonnegative integer. Define $[n]$ to be $\{1,\ldots, n\}$ if $n > 0$ and $\emptyset$ else. In addition, we set $[-1]:=\emptyset$.

\subsection{Compositions and partitions}
A \emph{composition} $\alpha$ of a nonnegative integer $n$, denoted by $\alpha \models n$, is a finite ordered list of positive integers $(\alpha_1,\ldots, \alpha_k)$ satisfying $\sum_{i=1}^k \alpha_i = n$.
For each $1 \le i \le k$, let us call $\alpha_i$ a \emph{part} of $\alpha$. And we call $k =: \ell(\alpha)$ the \emph{length} of $\alpha$ and $n =:|\alpha|$ the \emph{size} of $\alpha$. For convenience we define the empty composition $\emptyset$ to be the unique composition of size and length $0$.

Given $\alpha = (\alpha_1,\ldots,\alpha_k) \models n$ and $I = \{i_1 < \cdots < i_k\} \subset [n-1]$, 
let 
\begin{align*}
&\set(\alpha) := \{\alpha_1,\alpha_1+\alpha_2,\ldots, \alpha_1 + \alpha_2 + \cdots + \alpha_{k-1}\}, \\
&\comp(I) := (i_1,i_2 - i_1,\ldots,n-i_k).
\end{align*}
The set of compositions of $n$ is in bijection with the set of subsets of $[n-1]$ under the correspondence $\alpha \mapsto \set(\alpha)$ (or $I \mapsto \comp(I)$).

If $\alpha = (\alpha_1,\ldots, \alpha_k) \models n$ is such that $\alpha_1 \ge \cdots \ge \alpha_k$, then we say that $\alpha$ is \emph{partition} of $n$ and denote this by $\alpha \vdash n$. The partition obtained by sorting the parts of $\alpha$ in weakly decreasing order is denoted by $\tal$.

Let $\alpha = (\alpha_1,\ldots, \alpha_k) \models n$. We define the \emph{composition diagram} $\tcd(\alpha)$ of $\alpha$ as a left-justified array of $n$ boxes where the $i$th row from the top has $\alpha_i$ boxes for $1 \le i \le k$.
We also define the \emph{ribbon diagram} $\trd(\alpha)$ of $\alpha$ as the connected skew diagram without $2 \times 2$ boxes, such that the $i$th row from the bottom has $\alpha_i$ boxes. 
For $1 \le i < k$, the leftmost box of the $(i+1)$st row is stacked above the rightmost box of the $i$th row from the bottom.
For instance, if $\alpha = (1,3,2)$, then
\[
\tcd(\alpha) = 
\begin{array}{c}
\begin{ytableau}
~ \\
~ & ~ & ~ \\
~ & ~
\end{ytableau}
\end{array}
\quad \text{and} \quad
\trd(\alpha) = 
\begin{array}{c}
\begin{ytableau}
 \none & \none & ~ & ~ \\
 ~ & ~ & ~ \\
 ~
\end{ytableau}
\end{array}.
\]
A box is said to be in the $i$th row if it is in the $i$th row from the top, and in the $j$th column if it is in the $j$th column from the left.
We use $(i,j)$ to denote the box in the $i$th row and $j$th column.

For a composition $\alpha = (\alpha_1,\ldots,\alpha_k) \models n$, let $\alpha^\rmr$ denote the composition $(\alpha_k, \ldots, \alpha_1)$ and $\alpha^\rmc$ the unique composition satisfying $\set(\alpha^c) = [n-1] \setminus \set(\alpha)$, and let $\alpha^\rmt := (\alpha^\rmr)^\rmc = (\alpha^\rmc)^\rmr$. Note that $\trd(\alpha^\rmt)$ is obtained by reflecting $\trd(\alpha)$ along the diagonal. 

\subsection{The $0$-Hecke algebra and the quasisymmetric characteristic}
To begin with, we recall that the symmetric group $\SG_n$ is generated by simple transpositions $s_i := (i,i+1)$ with $1 \le i \le n-1$.
An expression for $\sigma \in \SG_n$ of the form $s_{i_1} \cdots s_{i_p}$ that uses the minimal number of simple transpositions is called a \emph{reduced expression} for $\sigma$. 
The number of simple transpositions in any reduced expression for $\sigma$, denoted by $\ell(\sigma)$, is called the \emph{length} of $\sigma$.

The $0$-Hecke algebra $H_n(0)$ is the $\C$-algebra generated by $\hpi_1,\ldots,\hpi_{n-1}$ subject to the following relations:
\begin{align}\label{eq: Hecke gen}
\begin{aligned}
\hpi_i^2 &= \hpi_i \quad \text{for $1\le i \le n-1$},\\
\hpi_i \hpi_{i+1} \hpi_i &= \hpi_{i+1} \hpi_i \hpi_{i+1}  \quad \text{for $1\le i \le n-2$},\\
\hpi_i \hpi_j &=\hpi_j \hpi_i \quad \text{if $|i-j| \ge 2$}.
\end{aligned}
\end{align}
For each $1 \le i \le n-1$, let $\opi_i := \hpi_i -1$.  Then $\{\opi_i \mid i = 1,\ldots,n-1\}$ is also a generating set of $H_n(0)$, which satisfies the following relations:
\begin{align*}
\opi_i^2 &= -\opi_i \quad \text{for $1\le i \le n-1$},\\
\opi_i \opi_{i+1} \opi_i &= \opi_{i+1} \opi_i \opi_{i+1}  \quad \text{for $1\le i \le n-2$},\\
\opi_i \opi_j &=\opi_j \opi_i \quad \text{if $|i-j| \ge 2$}.
\end{align*}

Pick up any reduced expression $s_{i_1}\cdots s_{i_p}$ for a permutation $\sigma \in \SG_n$. Then the elements $\hpi_{\sigma}$ and $\opi_{\sigma}$ of $H_n(0)$ are defined by
\[
\hpi_{\sigma} := \hpi_{i_1} \cdots \hpi_{i_p} \quad \text{and} \quad \opi_{\sigma} := \opi_{i_1} \cdots \opi_{i_p}.
\]
It is well known that these elements are independent of the choice of reduced expressions, and both $\{\hpi_\sigma \mid \sigma \in \SG_n\}$ and $\{\opi_\sigma \mid \sigma \in \SG_n\}$ are bases for $H_n(0)$.

Let $\calR(H_n(0))$ denote the $\Z$-span of the isomorphism classes of finite dimensional representations of $H_n(0)$. The isomorphism class corresponding to an $H_n(0)$-module $M$ will be denoted by $[M]$. The \emph{Grothendieck group} $\calG_0(H_n(0))$ is the quotient of $\calR(H_n(0))$ modulo the relations $[M] = [M'] + [M'']$ whenever there exists a short exact sequence $0 \ra M' \ra M \ra M'' \ra 0$. The irreducible representations of $H_n(0)$ form a free $\Z$-basis for $\calG_0(H_n(0))$. Let
\[
\calG := \bigoplus_{n \ge 0} \calG_0(H_n(0)).
\]
According to \cite{79Norton}, there are $2^{n-1}$ distinct irreducible representations of $H_n(0)$. They are naturally indexed by compositions of $n$. Let $\bfF_{\alpha}$ denote the $1$-dimensional $\C$-vector space corresponding to the composition $\alpha \models n$, spanned by a vector $v_{\alpha}$. 
For each $1\le i \le n-1$, define an action of the generator $\hpi_i$ of $H_n(0)$ as follows:
\[
\hpi_i(v_\alpha) = \begin{cases}
0 & i \in \set(\alpha),\\
v_\alpha & i \notin \set(\alpha).
\end{cases}
\]
Then $\bfF_\alpha$ is an irreducible $1$-dimensional $H_n(0)$-representation.

In the following, let us review the connection between $\calG$ and quasisymmetric functions.
Quasisymmetric functions are power series of bounded degree in variables $x_{1},x_{2},x_{3},\ldots$  with coefficients in $\Z$, which are shift invariant in the sense that the coefficient of the monomial $x_{1}^{\alpha _{1}}x_{2}^{\alpha _{2}}\cdots x_{k}^{\alpha _{k}}$ is equal to the coefficient of the monomial $x_{i_{1}}^{\alpha _{1}}x_{i_{2}}^{\alpha _{2}}\cdots x_{i_{k}}^{\alpha _{k}}$ for any strictly increasing sequence of positive integers $i_{1}<i_{2}<\cdots <i_{k}$ indexing the variables and any positive integer sequence $(\alpha _{1},\alpha _{2},\ldots ,\alpha _{k})$ of exponents.

$\Qsym$ is a graded $\Z$-algebra, decomposing as
\[
\Qsym =\bigoplus _{n\geq 0} \Qsym_{n},
\]
where $\Qsym_{n}$ is the $\Z$-span of all quasisymmetric functions that are homogeneous of degree $n$.

Given a composition $\alpha$, the \emph{fundamental quasisymmetric function} $F_\alpha$ is defined by $F_\emptyset = 1$ and
\[
F_\alpha = \sum_{\substack{1 \le i_1 \le \cdots \le i_k \\ i_j < i_{j+1} \text{ if } j \in \set(\alpha)}} x_{i_1} \cdots x_{i_k}.
\]
For every nonnegative integer $n$, it is known that $\{F_\alpha \mid \alpha \models n\}$ is a basis for $\Qsym_n$.

It was shown in \cite{96DKLT} that, when $\calG$ is equipped with induction product, the linear map
\begin{align*}
\ch : \calG \ra \Qsym, \quad [\bfF_{\alpha}] \mapsto F_{\alpha},
\end{align*}
called \emph{quasisymmetric characteristic}, is a ring isomorphism.
It follows from the definition of Grothendieck group that 
if we have a short exact sequence $0 \ra M' \ra M \ra M'' \ra 0$ of finite dimensional $H_n(0)$-modules then
\[
\ch([M]) = \ch([M']) + \ch([M'']).
\]

\subsection{Projective indecomposable $H_n(0)$-modules}\label{subsec: PIM}

For $I \subseteq [n-1]$, let $I^\rmc := [n-1] \setminus I$. The \emph{parabolic subgroup} $\SG_{n,I}$ is the subgroup of $\SG_n$ generated by $\{s_i \mid i\in I\}$. We denote by $w_0(I)$ the longest element in $\SG_{n,I}$. 
Norton~\cite{79Norton} decomposed the regular representation of $H_n(0)$ into the direct sum of $2^{n-1}$ indecomposable submodules $\calP_I$ $(I \subseteq [n-1])$, which are defined by
\begin{align*}
\calP_I := H_n(0)\cdot \opi_{w_0(I)}\hpi_{w_0(I^\rmc)}.
\end{align*}
Let $\rmtop(\calP_I)$ denote the \emph{top} of $\calP_I$, that is, $\rmtop(\calP_I) := \calP_I / \rad \; \calP_I$.
Here $\rad \; \calP_I$ is the \emph{radical} of $\calP_I$, the intersection of maximal submodules of $\calP_I$.
It is known that, for each $I \subseteq [n-1]$, $\rmtop(\calP_I)$ is isomorphic to $\bfF_\alpha$ with $\set(\alpha) = I$.
The set $\{\calP_I \mid I \subseteq [n-1]\}$ is a complete list of non-isomorphic projective indecomposable $H_n(0)$ modules.

The $H_n(0)$-module $\calP_I$ can be described in a combinatorial way (see \cite{06HNT, 97KT}). To explain this, let us introduce \emph{standard ribbon tableaux} and an $H_n(0)$-action on them.

\begin{definition}
For $\alpha \models n$, a \emph{standard ribbon tableau} of shape $\alpha$ is a filling of $\trd(\alpha)$ by $1,2,\ldots,n$ without repetition such that 
entries in each row increase from left to right and entries in each column increase from top to bottom.
\end{definition}
Let $\SRT(\alpha)$ denote the set of all standard ribbon tableaux of shape $\alpha$. 
We denote by $\bfP_\alpha$ the $\C$-span of $\SRT(\alpha)$.
The $H_n(0)$-action on $\bfP_\alpha$ is defined by
\begin{align*}
\opi_i \cdot T = \begin{cases}
-T, & \text{if $i$ is in a higher row of $T$ than $i+1$},\\
0, & \text{if $i$ is in the same row of $T$ as $i+1$},\\
s_i \cdot T, & \text{if $i$ is in a lower row of $T$ than $i+1$},
\end{cases}
\end{align*}
for $1\le i \le n-1$ and $T \in \SRT(\alpha)$. 
Here $s_i \cdot T$ is obtained from $T$ by swapping $i$ and $i+1$. 
The following result is due to Huang~\cite{16Huang}.
\begin{theorem}{\rm (\cite[Theorem 3.3.(i)]{16Huang})}\label{Thm22}
Let $\alpha \models n$. Then $\bfP_\alpha$ is isomorphic to $\calP_{\set(\alpha)}$ as an $H_n(0)$-module.
\end{theorem}

\subsection{Standard permuted composition tableaux and an $H_n(0)$-module arising from them}\label{subsec: SPCT}
Standard permuted composition tableaux and an $H_n(0)$-action on them were introduced and studied intensively by Tewari and van Willigenburg~\cite{15TW, 19TW}. 

Let $w = w_1 w_2 \cdots w_\ell$ be a word without repeated entries in the alphabet $\{1 < 2 < \cdots\}$.
The \emph{standardization} of $w$ is the unique permutation $\sigma \in \SG_\ell$ satisfying that $\sigma(i) > \sigma(j)$ if and only if $w_{i} > w_{j}$ for all $1 \le i < j \le \ell$.

\begin{definition}\label{def: PCT}
Given $\alpha \models n$ and $\sigma \in \SG_{\ell(\alpha)}$, a \emph{permuted composition tableau} ($\PCT$) of shape $\alpha$ and type $\sigma$ is a filling $\tau$ of $\tcd(\alpha)$ with positive integers such that the following conditions hold:
\begin{enumerate}[label = (\arabic*)]
\item The entries in the first column are all distinct and the standardization of the word obtained by reading the first column from top to bottom is $\sigma$.
\item The entries along the rows decrease weakly when read from left to right.
\item If $i<j$ and $\tau(i,k) > \tau(j,k+1)$, then $\tau(i,k+1)$ exists and $\tau(i,k+1) > \tau(j,k+1)$.
\end{enumerate}
The condition (3) is called the \emph{triple condition}.
In particular, if the entries of $\tau$ are weakly less than $n$ and all distinct, then we call $\tau$ a \emph{standard permuted composition tableau} ($\SPCT$).
\end{definition}

Let $\PCT^\sigma(\alpha)$ (resp. $\SPCTsa$) denote the set of all $\PCT$s (resp. $\SPCT$s) of shape $\alpha$ and type $\sigma$. 
We remark that when $\sigma = \id$, permuted composition tableaux are identical to reverse composition tableaux in~\cite{15TW}, thus $\SPCT^{\id}(\alpha)$ is nothing but $\SRCT(\alpha)$ therein. 
And, we write $\tab_{i,j}$ for $\tab(i,j)$ for  $\tab \in \PCT^\sigma(\alpha)$ and $(i,j) \in \tcd(\alpha)$.
For an SPCT $\tab$, we use the notation $\tau^{-1}(i)$ to denote the box in $\tau$ filled with $i$. 

Let $\tau \in \SPCTsa$. An integer $1 \le i \le n-1$ is a \emph{descent} of $\tau$ if $i+1$ lies weakly right of $i$ in $\tau$. Denote by $\Des(\tau)$ the set of all descents of $\tau$ and set $\comp(\tau) := \comp(\Des(\tau))$.
And, for $1 \leq i < j \leq n$, we say that $i$ and $j$ are \emph{attacking} (in $\tab$)
if either
\begin{enumerate}[label = (\roman*)]
\item $i$ and $j$ are in the same column in $\tau$, or
\item $i$ and $j$ are in adjacent columns in $\tau$, with $j$ positioned lower-right of $i$.
\end{enumerate}

Let $\bfSsa$ be the $\C$-span of $\SPCTsa$.
The $H_n(0)$-action on $\bfSsa$ is defined by
\begin{align}\label{eq: Hecke action on SPCT}
\hpi_i \cdot \tau = \begin{cases}
\tau & \text{if $i \notin \Des(\tab)$},\\
0 & \text{if $i \in \Des(\tab)$, $i$ and $i+1$ are attacking},\\
s_i \cdot \tab & \text{if $i \in \Des(\tab)$, $i$ and $i+1$ are nonattacking}
\end{cases}
\end{align}
for $1 \le i \le n-1$ and $\tau \in \SPCTsa$.
Here $s_i \cdot \tab$ is obtained from $\tab$ by swapping $i$ and $i+1$. 

\begin{theorem}\cite[Theorem 3.1]{19TW}\label{thm: Hecke action on SPCT}
Given $\alpha \models n$, \eqref{eq: Hecke action on SPCT} induces an $H_n(0)$-action on $\bfSsa$.
\end{theorem}

Let $\alpha$ be a composition whose largest part is $\alphamax$ and $\tau$ an SPCT of shape $\alpha$ and type $\sigma$. For $1 \le i \le \alphamax$, we define the \emph{$i$th column word} $w^i(\tau)$ of $\tau$ to be the word obtained from $\tab$ by reading the entries in the $i$th column from top to bottom. The \emph{column word of $\tab$} is the word
$w^1(\tau) \  w^2(\tau) \cdots w^{\alphamax}(\tau)$, which is considered as a permutation in $\SG_n$ written in one-line notation. Denote this word by $\col_\tau$.

The \emph{standardized $i$th column word} of $\tau$, denoted by $\rmst(w^i(\tau))$, is the permutation $\sigma \in \SG_{\ell(w^i(\tau))}$ uniquely determined by the condition: for $1 \le j < j' \le \ell(w^i(\tau))$,
\begin{align*}
\sigma(j) > \sigma(j') \quad \text{if and only if} \quad  w^i(\tau)_{j} > w^i(\tau)_{j'}.    
\end{align*}
Here $\ell(w^i(\tau))$ is the length of the word $w^i(\tau)$. The \emph{standardized column word} of $\tau$, denoted by $\rmst(\tau)$, is the word obtained by concatenating $\rmst(w^i(\tau))$'s:
\[
\rmst(\tau) = \rmst(w^1(\tau)) \  \rmst(w^2(\tau)) \cdots \rmst(w^{\alphamax}(\tau)).
\]
For instance, if $\tau = \begin{array}{c}
\begin{ytableau}
4 & 1 \\
6 & 5 & 3 \\
2 
\end{ytableau}
\end{array}$,
then $\rmst(\tau) = 231 12 1$.

Recall that the equivalence relation $\sim_{\alpha}$ on $\SPCTsa$ defined by
\[
\tau_1 \sim_{\alpha} \tau_2 \quad \text{if and only if} \quad \rmst(\tau_1) = \rmst(\tau_2)\quad \text{for $\tau_1,\tau_2 \in \SPCTsa$}
\]
was introduced in~\cite[Subsection 3.1]{19TW}.
Let $\mathcal{E}^\sigma(\alpha)$ be the set of all equivalence classes under $\sim_{\alpha}$.

Next, we introduce the notion of \emph{source tableaux} and \emph{sink tableaux}.

\begin{definition}(\cite{19TW})
\label{def: source and sink}\hfill
\begin{enumerate}
\item An $\SPCT$ $\tau$ is said to be a \emph{source tableau} if, for every $i \notin \Des(\tau)$ where $i \neq n$, we have that $i+1$ lies to the immediate left of $i$.
\item An $\SPCT$ $\tau$ is said to be a \emph{sink tableau} if, for every $i \in \Des(\tau)$, we have that $i$ and $i+1$ are attacking.
\end{enumerate}
\end{definition}

\begin{lemma}{\rm (\cite[Lemma 3.7]{19TW})}\label{lem: unique source and sink}
For each $E \in \mathcal{E}^{\sigma}(\alpha)$, there is a unique source tableau and a unique sink tableau in $E$.
\end{lemma}

\section{The indecomposability of $\bfSsa$}\label{sec: Indecomposability}

Hereafter we assume that $n$ is a positive integer.
Let $\alpha \models n$ and $\sigma \in \SG_{\ell(\alpha)}$.
The purpose of this section is to classify all $H_n(0)$-modules $\bfS^{\sigma}_{\alpha}$ that are indecomposable.

To begin with, we recall the decomposition of $\bfS^\sigma_\alpha$ due to Tewari and van Willigenburg~\cite[Lemma 3.6]{19TW}.
For each $E \in \calE^\sigma(\alpha)$, let $\bfS_{\alpha,E}^{\sigma}$ be 
the $H_n(0)$-submodule of $\bfS_\alpha^{\sigma}$ whose underlying space is the $\C$-span of $E$.
As $H_n(0)$-modules, we have
\begin{align}\label{eq: decomp by eq rel}
\bfS_\alpha^{\sigma} \cong \bigoplus_{E \in \mathcal{E}^\sigma(\alpha)} \bfS_{\alpha,E}^{\sigma}.
\end{align}
In case where $\sigma = \id$, K\"onig~\cite{19Konig} proved that 
$\bfS^{\id}_{\alpha,E}$ is indecomposable for every $E \in \mathcal{E}^{\id}(\alpha)$.
It is quite interesting to note that this result holds for all $\sigma \in \SG_{\ell(\alpha)}$.
More precisely, by a slight modification of the arguments in~\cite{19Konig},
we can state the following theorem.

\begin{theorem}{\rm (cf.~\cite[Theorem~4.11]{19Konig})}
\label{thm: eq class indecomp}
Let $\alpha \models n$ and $\sigma \in \SG_{\ell(\alpha)}$. 
For every $E \in \mathcal{E}^{\sigma}(\alpha)$, 
the $H_n(0)$-module $\bfS^{\sigma}_{\alpha,E}$ is indecomposable.
\end{theorem}

Since the proof of Theorem~\ref{thm: eq class indecomp} is mostly overlapped with~\cite{19Konig},
it will not be contained in this paper.
Instead, in Appendix~\ref{sec: indecomp}, we write down 
arguments necessary for the proof which are revised in accordance with our situation.  

We say that $\bfS^{\sigma}_{\alpha}$ is {\em SPCT-cyclic} 
if it is cyclically generated by a single SPCT.
Combining~\eqref{eq: decomp by eq rel} with Theorem~\ref{thm: eq class indecomp} shows that
$\bfS^{\sigma}_{\alpha}$ is SPCT-cyclic if and only if it is indecomposable.
In case where $\sigma = \id$, the classification of these modules was done by Tewari and van Willigenburg~\cite[Theorem 7.6]{15TW}.
We extend their result to all $\sigma \in \SG_{\ell(\alpha)}$.

Given a composition $\alpha = (\alpha_1, \alpha_2, \cdots, \alpha_\ell)$ and $\sigma \in \SG_{\ell}$,
we say that $\alpha$
\emph{is compatible with} $\sigma$ if $\alpha_i \geq \alpha_j$ whenever $i<j$ and $\sigma(i) > \sigma(j)$.
In \cite{13HMR}
Haglund {\it et al.} introduced a simple characterization for pairs $(\sigma, \alpha)$ with the property that
$\SPCTsa \neq \emptyset$. 
Although stated for
\emph{permuted basement semistandard augmented fillings}, one can see that their characterization can be applied to SPCTs directly.

\begin{proposition} {\rm (\cite[Proposition 14]{13HMR})}
\label{Prop: comparibility}
For a composition $\alpha$ and a permutation $\sigma$ in $\SG_{\ell(\alpha)}$,
  $\SPCTsa$ is nonempty if and only if $\alpha$ is compatible with $\sigma$.
\end{proposition}

\begin{definition}\label{def: simple}
  Let $\alpha = (\alpha_1, \alpha_2, \ldots, \alpha_\ell)$ be compatible with $\sigma$ for $\sigma \in \SG_{\ell}$.
\begin{enumerate}
  \item[(a)] For $i < j$,
  if $\sigma(i) < \sigma(j)$ and $\alpha_i \geq \alpha_j \geq 2$,
  then $(i, j)$ is called {\it a permutation-ascending composition-descending (PACD) pair attached to the pair $(\alpha \, ; \sigma)$}.
  \item[(b)] We say that $\alpha$ is \emph{$\sigma$-simple} if every PACD pair $(i,j)$ attached to $(\alpha \, ; \sigma)$ satisfies one of the following conditions:
      \begin{enumerate}
        \item[{\bf C1.}] There exists $i < k < j$ such that
        $\sigma(i) < \sigma(k) < \sigma(j)$ and $\alpha_k = \alpha_j -1$.
        \item[{\bf C2.}] There exists $k > j$ such that
         $\sigma(i) < \sigma(k) < \sigma(j)$ and $\alpha_k = \alpha_j$.
      \end{enumerate}
\end{enumerate}
\end{definition}
If $\sigma = \id$, then Definition~\ref{def: simple}(b) agrees with the definition of simple compositions introduced in \cite{15TW}.
\begin{example}
Let $\sigma = 2134 \in \SG_4$.
  \begin{enumerate}[label = {\rm (\alph*)}]
    \item The composition $\alpha = (1,2,1,3)$ is not compatible with $\sigma$ since $\sigma(1)> \sigma(2)$ but $\alpha_1 < \alpha_2$.
    \item The composition $\alpha = (1,1,2,3)$ is $\sigma$-simple since there are no PACD pairs attached to $(\alpha\, ; \sigma)$.
    \item The composition $\alpha = (3,3,1,2)$ is $\sigma$-simple since $(1,4)$ is a unique PACD pair attached to  $(\alpha \, ; \sigma)$ and $3$ satisfies the condition {\bf C1}. 
    \item The composition $\alpha = (3,1,2,2)$ is not $\sigma$-simple since
    the PACD pair $(1,4)$ attached to $(\alpha\, ; \sigma)$ satisfies neither {\bf C1} nor {\bf C2}.
  \end{enumerate}
\end{example}

For a composition $\alpha$ which is compatible with $\sigma$,
we say that a part $\alpha_j$ \emph{has a removable node {\rm (}with respect to $\sigma${\rm )}} if
either $\sigma(j) = 1$ or the following three conditions hold:
\begin{enumerate}[label = {\bf R\arabic*.}]
  \item $\alpha_j \geq 2$,
  \item there exists no $ k < j$ such that
  $\sigma(k) < \sigma(j)$ and $\alpha_k = \alpha_j -1$, and
  \item there exists no $k > j$ such that
  $\sigma(k) < \sigma(j)$ and $\alpha_k = \alpha_j$.
\end{enumerate}
In this case, we call the box $(j, \alpha_j) \in \tcd(\alpha)$ a \emph{removable node {\rm (}with respect to $\sigma${\rm )}}.

\begin{example}\label{exam: removable}
Let $\alpha = (4,1,2,2)$ and $\sigma = 2134$.
In the diagram below, the dotted boxes are all removable nodes.
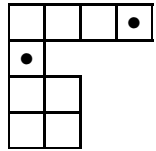
\begin{align*}
{\small \begin{ytableau}
  \empty & \empty & \empty & \bullet\\
   \bullet \\
  \empty & \empty \\
  \empty & \empty
\end{ytableau}}
\end{align*}
\end{example}
\vskip 1mm

In the rest of this subsection,
we always assume that $\alpha$ is compatible with $\sigma$ unless otherwise stated.
Let us provide two vital properties of $\sigma$-simple compositions related to removable nodes.

\begin{lemma}\label{lem_removable1}
  Let  $\alpha = (\alpha_1, \alpha_2, \ldots, \alpha_\ell)$ be a $\sigma$-simple composition.
  The part $\alpha_j$ for $1 \leq j \leq \ell$ has a removable node if and only if either $\sigma(j) =1$ or the following are satisfied:
  \begin{enumerate}[label = {\rm (\roman*)}]
    \item for all $i < j$ with $\sigma(i) < \sigma(j)$,
    we have $\alpha_i \leq \alpha_j -2$, and
    \item for all $i > j$ with $\sigma(i) < \sigma(j)$,
    we have $\alpha_i \neq \alpha_j $.
  \end{enumerate}
\end{lemma}
\begin{proof}
We first suppose that the part $\alpha_j$ has a removable node.
The assertion is clear when $\sigma(j) =1$,
so we assume that $\sigma(j) \neq 1$.
Note that the assertion (ii) is identical to the condition {\bf R3}.
In addition, the condition {\bf R2} says that there exists no $i<j$ such that $\sigma(i) < \sigma(j)$ and $\alpha_i = \alpha_j -1$.
Hence, for all $i<j$ satisfying $\sigma(i) < \sigma(j)$ and $\alpha_i < \alpha_j$,
we have $\alpha_i \leq \alpha_j -2$.
Now we claim that there is no $i < j$ such that
$\sigma(i) < \sigma(j)$ and $\alpha_i \geq \alpha_j$, that is, there is no PACD pair $(i,j)$ attached to $(\alpha \, ; \sigma)$.
Suppose on the contrary that there exists such an $i$.
Since $\alpha$ is $\sigma$-simple and
$(i,j)$ is a PACD pair,
$(i,j)$ satisfies either {\bf C1} or {\bf C2}. In the former case,
we have $k$ satisfying $i < k < j$, $\sigma(i) < \sigma(k) < \sigma(j)$, and $\alpha_k = \alpha_j-1$, which contradicts our assumption that $\alpha_j$ has a removable node. The latter case does not occur due to the assertion (ii).

Now let us prove the converse.
The conditions {\bf R2} and {\bf R3} follow from (i) and (ii), respectively.
We claim that the index $j$ satisfies either $\sigma(j) = 1$ or {\bf R1}.
To show this,
assume that $\alpha_j = 1$. For $i<j$, $\sigma(i) > \sigma(j)$ by (i) and for $i>j$, $\sigma(i) > \sigma(j)$ by (ii) together with the $\sigma$-compatibility. Thus $\sigma(j) = 1$, as required.
\end{proof}

Given $\sigma \in \SG_{\ell}$,
we denote by $\sigma_{\downarrow}$ the permutation in $\SG_{\ell-1}$
obtained from $\sigma$ by
removing $1$ and then subtracting $1$ from each remaining part.
For a composition $\alpha$ of length $\ell$,
when $\alpha_m$ has a removable node, set
\[
\bsig_{(\alpha;m)} :=
\begin{cases}
\sigma & \text{ if } \alpha_{m} > 1, \\
\sigma_{\downarrow} & \text{ if } \alpha_{m} = 1.
\end{cases}
\]
In case where $\alpha$ and $m$ are manifest from the context, we conventionally drop the subscript from $\bsig_{(\alpha;m)}$.

\begin{lemma}\label{lem_removable2}
Let $\alpha = (\alpha_1, \alpha_2, \cdots, \alpha_\ell)$ be a $\sigma$-simple composition. If the part $\alpha_m$ has a removable node, then the composition $(\alpha_1, \cdots, \alpha_m -1, \cdots, \alpha_\ell)$
is $\bsig$-simple.
\end{lemma}
\begin{proof}
For simplicity, set $\hat{\alpha}:= (\alpha_1, \cdots, \alpha_m -1, \cdots, \alpha_\ell)$.
We first notice that
$\hat{\alpha}$ is compatible with $\bsig$
since there is no $k > m$ such that $\sigma(k) < \sigma(m)$ and $\alpha_k = \alpha_m$.

First, let us deal with the case where $\sigma(m) = 1$.
Let $(i,j)$ be a PACD pair attached to $(\hat{\alpha} \, ; \bsig)$.
Notice that $j$ cannot be $m$ and $\alpha_i \geq \hat{\alpha}_i$ for all $i$.
So one can see that $(i, j)$ is also a PACD pair attached to $(\alpha \, ; \sigma)$.
Since $\alpha$ is $\sigma$-simple,
there exists a positive integer $k$ satisfying {\bf C1} or {\bf C2}.
Hence, to show that $\hat{\alpha}$ is $\bsig$-simple,
it is enough to see that $k \neq m$.
But this is clear because $\sigma(i) < \sigma(k) < \sigma(j)$.

Next, let us deal with the case where $\sigma(m) \neq 1$.
If $\alpha_{m} = 1$, then $\sigma(m) = 1$ by the definition of removable nodes.
So we assume that $\alpha_m \ge 2$.
Let $(i,j)$ be any PACD pair attached to $(\hat{\alpha} \, ; \sigma)$,
We will prove that
there exists a positive integer $k$  satisfying one of {\bf C1} and {\bf C2} in Definition \ref{def: simple}(b).
To begin with,
we note that Lemma \ref{lem_removable1} implies that
$\hat{\alpha}_i < \hat{\alpha}_m$ for all $1 \leq i < m$ with $\sigma(i) < \sigma(m)$.
Thus $(i, m)$ cannot be a PACD pair attached to $(\hat{\alpha} \, ; \sigma)$.
We have the following three cases.
\smallskip

  {\it Case 1: $1 \leq i < j < m$.}
  Since $\hat{\alpha_i} = \alpha_i$ and $\hat{\alpha}_j = \alpha_j$,
  $(i,j)$ is also a PACD pair attached to $(\alpha \, ; \sigma)$.
  Since $\alpha$ is $\sigma$-simple,
  there exists a positive integer $k$ satisfying either  {\bf C1} or {\bf C2}.
  When $k$ satisfies {\bf C1}, $\hat{\alpha_k} = \alpha_k$
  and thus the PACD pair $(i,j)$ attached to $(\hat{\alpha} \, ; \sigma)$ also satisfies {\bf C1}.
  On the other hand,
  when $k$ satisfies {\bf C2},
  we have $k>j$ such that $\sigma(i) < \sigma(k) < \sigma(j)$ and $\alpha_k = \alpha_j$.
  To verify that the PACD pair $(i,j)$ attached to $(\hat{\alpha} \, ; \sigma)$ also satisfies {\bf C2},
  we need to establish that $k \neq m$.
  Since $\alpha_m$ has a removable node,
  by Lemma \ref{lem_removable1},
  we have that
  $\alpha_i \leq \alpha_m-2$
  for all $i<m$ with $\sigma(i) < \sigma(m)$.
  However, if $k = m$, then
  $\alpha_i \geq \alpha_j = \alpha_m$
  which contradicts our assumption that $\alpha_m$ has a removable node.
\smallskip

  {\it Case 2: $m \leq i < j \leq \ell$.}
  We can prove the assertion as in {\it Case 1}.
\smallskip

  {\it Case 3: $1 \leq i \leq m < j \leq \ell$.}
  Since $(i,j)$ is also a PACD pair attached to $(\alpha \, ; \sigma)$,
  there exists a positive integer $k$ satisfying either {\bf C1} or {\bf C2}.
  When $k$ satisfies {\bf C1},
  we have $i < k < j$ such that $\sigma(i) < \sigma(k) < \sigma(j)$ and $\alpha_k = \alpha_j -1$.
  If  $\sigma(i)<\sigma(m)$,
  then $k$ cannot be $m$ as $\alpha_i \leq \alpha_m-2$ and $\alpha_i > \alpha_k$.
  If $\sigma(i) > \sigma(m)$,
  then $k$ cannot be $m$ since $\sigma(k) > \sigma(i) > \sigma(m) $.
  In any case,
  we can say that the PACD pair $(i,j)$ attached to $(\hat{\alpha} \, ; \sigma)$
  also satisfies {\bf C1}.
  On the other hand,
  when $k$ satisfies {\bf C2},
  it is obvious that
  the PACD pair $(i,j)$ attached to $(\hat{\alpha} \, ; \sigma)$
  also satisfies {\bf C2}.
  \smallskip

Combining {\it Case 1}, {\it Case 2}, and {\it Case 3} with the equality $\hat{\alpha}_k = \alpha_k$,
we conclude that $\hat{\alpha}$ is a $\sigma$-simple composition.
\end{proof}

For every composition $\alpha = (\alpha_1,\ldots, \alpha_{\ell})$ and $\sigma \in \SG_\ell$,
one can construct a source tableau in $\SPCTsa$ in the following way:
For $1 \leq i \leq \ell$,
fill the $\alpha_{\sigma^{-1}(i)}$th row with
\[
\sum_{j = 0}^{i-1} \alpha_{\sigma^{-1}(j)} + 1 \, ,  \sum_{j = 0}^{i-1} \alpha_{\sigma^{-1}(j)} + 2 \, , \ldots \, , \sum_{j = 0}^{i-1} \alpha_{\sigma^{-1}(j)} + \alpha_{\sigma^{-1}(i)}
\]
in decreasing order from left to right.
Here $\alpha_{\sigma^{-1}(0)}$ is set to be 0.
\begin{definition}\label{def: canonical tableau}
The tableau defined as above is called \emph{the canonical source tableau of shape $\alpha$ and type $\sigma$}, and is denoted by $\tauC$.
\end{definition}
When $\alpha_{\sigma^{-1}(\ell)} \ge 2$, define a new filling $\htauC$ of $\tcd(\alpha)$ in the following steps:
Set $\beta := (\alpha_1, \cdots, \alpha_{\sigma^{-1}(\ell)} -1, \cdots, \alpha_{\ell})$.
\begin{enumerate}[label = {\it Step \arabic*.}]
  \item Fill the composition diagram $\tcd(\beta)$ so that the resulting tableau is the canonical source tableau of shape $\beta$ and type $\sigma$.
  \item Add 1 to every entry on the filling obtained in {\it Step 1}.
  \item Add a box with entry 1 on the rightmost of the ${\sigma^{-1}(\ell)}$th row of the filling obtained in {\it Step 2}. Set $\htauC$ to be the resulting filling.
\end{enumerate}
For example, let $\alpha = (1,3,2,4)$ and $\sigma = 1324$.
Then we have 
\[
 \tauC =
 \begin{array}{l}
 {\small \begin{ytableau}
  1 \\
  6 & 5 & 4 \\
  3 & 2 \\
  10 & 9 & 8 & 7
\end{ytableau}}
\end {array} \quad  \text{and} \quad
 \htauC =
 \begin{array}{l}
 {\small 
 \begin{ytableau}
  2 \\
  7 & 6 & 5 \\
  4 & 3 \\
  10 & 9 & 8 & 1
\end{ytableau}}
\end{array} .
\]
As one can see in this example,
$\htauC$ is not necessarily an SPCT.
However, the following lemma shows that under a suitable condition, 
$\htauC$ is not only an SPCT but also a source tableau.

\begin{lemma}\label{lem : construct source}
With the above notation,
suppose that $\sigma^{-1}(1) < \sigma^{-1}(\ell)$
and $\alpha_{\sigma^{-1}(1)} \geq \alpha_{\sigma^{-1}(\ell)} \geq 2$.
If the part $\alpha_{\sigma^{-1}(\ell)}$ has a removable node,
then $\htauC$ is a source tableau in $\SPCTsa$.
\end{lemma}
\begin{proof}
We begin with the observation that ignoring the rightmost box in the  ${\sigma^{-1}(\ell)}$th row, $\tauC$ and $\htauC$ have the same standardization.
To verify that $\htauC$ is in $\SPCTsa$,
it suffices to show that
the rightmost box on the ${\sigma^{-1}(\ell)}$th row, which is filled with $1$, obeys the triple condition.
Equivalently, it suffices to show that the following two events do not occur in $\htauC$:
\[
\begin{array}{l}
\begin{ytableau}
a \\
\none[\vdots]\\
&1
\end{ytableau}
\end{array}
\quad \text{or} \quad
\begin{array}{l}
\begin{ytableau}
a & 1\\
\none & \none[\vdots]\\
\none&b
\end{ytableau}
\end{array}
\quad \text{with $a>b$}
\]
This can be derived from the assumption that $\alpha_{\sigma^{-1}(\ell(\alpha))}$ has a removable node,
which guarantees that there exists no part $\alpha_i$ above  $\alpha_{\sigma^{-1}(\ell(\alpha))}$
satisfying  $\alpha_i = \alpha_{\sigma^{-1}(\ell(\alpha))} -1$ and
there exists no part $\alpha_j$ below  $\alpha_{\sigma^{-1}(\ell(\alpha))}$
satisfying  $\alpha_i = \alpha_{\sigma^{-1}(\ell(\alpha))}$.

Since the rightmost box in the $\sigma^{-1}(1)$th row is filled with 2
and $\alpha_{\sigma^{-1}(1)} \geq \alpha_{\sigma^{-1}(\ell(\alpha))}$,
we have that $1 \in \Des(\htauC)$.
For $1< i < |\alpha|$,
suppose that $i+1$ does not lie to the immediate left of $i$.
Then $i$ and $i+1$ cannot be placed on the same row.
Moreover,
$i$ lies on the leftmost box of a row
and $i+1$ lies on the rightmost box of a row because of the construction of $\htauC$.
This implies that $i \in \Des(\htauC)$.
As a consequence, $\htauC$ is a source tableau.
\end{proof}

For $\tau \in \SPCTsa$,
suppose that $\tau_{m, \alpha_m} = 1$.
If $\alpha_i = \alpha_m$ for some $i > m$ with $\sigma(i) < \sigma(m)$, then
$\tab_{i, \alpha_m -k+1} > \tab_{m, \alpha_m-k}$ for $1 \leq k \leq \alpha_m -1$, thus $\tab_{i,1} > \tab_{m,1}$.
But this cannot occur by the condition $\sigma(i) < \sigma(m)$, so $\alpha_i \neq \alpha_m$ for $i > m$ with $\sigma(i) < \sigma(m)$.
For $i < m$ with $\sigma(i) < \sigma(m)$, one has that $\alpha_i \neq \alpha_m -1$ by the triple condition.
Consequently, we can conclude that the box $(m, \alpha_m)$ is a removable node.

On the other hand,
when the part $\alpha_m$ has a removable node, one can construct an SPCT of shape $\alpha$ and type $\sigma$ 
such that $(m, \alpha_m)$ is filled with $1$ as follows:
Let $\hat{\alpha} = (\alpha_1, \cdots, \alpha_m -1, \cdots, \alpha_\ell)$.
Pick up any SPCT $\hat{\tau}$ of shape $\hat{\alpha}$ and type $\bsig$. 
Add $1$ to every entry in $\hat{\tau}$ and then place a new box filled with $1$ at the position $(m, \alpha_m)$.
We denote the resulting filling by $\hat\tau'$.
It is straightforward to verify that $\hat\tau' \in \SPCTsa$.

As seen in Example~\ref{exam: removable},
$\alpha$ may have more than one removable nodes.
So, for $\tau, \tau' \in \SPCTsa$, the position of $1$ in $\tau$ may be different from that in $\tau'$.
However, if $\alpha$ is $\sigma$-simple, then the position of $1$ in any source tableau is determined in a unique way.

\begin{proposition}\label{prop: row of 1 in source tab}
Let $\alpha$ be a composition and $\sigma \in \SG_{\ell(\alpha)}$.
Let $\tau_0$ be a source tableau in $\SPCTsa$.
If $\alpha$ is a $\sigma$-simple composition,
then the box filled with $1$ lies on the rightmost of the $\sigma^{-1}(1)$th row in $\tau_0$.
\end{proposition}

To prove this proposition,
we need the following lemma.
\begin{lemma}\label{lem: type dependence}
Let $\alpha = (\alpha_1, \alpha_2, \ldots, \alpha_\ell)$ be a composition and $\sigma \in \SG_{\ell}$.
Let $\tau \in \SPCTsa$ and set $m := {\rm min}(\alpha_i, \alpha_j)$.
If $\sigma(i) < \sigma(j)$ and $\tau_{i,m} > \tau_{j,m}$, then $i<j$.
\end{lemma}
\begin{proof}
Suppose on the contrary that $i > j$.
Let $t_k = \tau_{i,m-k+1}$ and $s_k = \tau_{j,m-k+1}$ for each $1 \leq k \leq m$:
\[
\begin{tikzpicture}[scale=.55]
\draw[-] (0,0)--(5,0)--(5,1)--(0,1)--(0,0);
\draw[-] (3,0)--(3,1);
\draw[-] (4,0)--(4,1);
\draw[-] (1,0)--(1,1);
\draw[-] (5,0)--(5,1);
\draw[-] (0,-1.5+0.3)--(5,-1.5+0.3)--(5,-2.5+0.3)--(0,-2.5+0.3)--(0,-1.5+0.3);
\draw[-] (3,-1.5+0.3)--(3,-2.5+0.3);
\draw[-] (1,-1.5+0.3)--(1,-2.5+0.3);
\draw[-] (4,-1.5+0.3)--(4,-2.5+0.3);
\node (0) at (-2, -1.9+0.3) {row $i$ :};
\node (1) at (-2, .55) {row $j$ :};
\node (2) at (2.1,-.6+0.15) {$\vdots$};
\node (3) at (0.5, .45) {$s_m$};
\node (3) at (2, .45) {$\cdots$};
\node (4) at (3.5, .45) {$s_2$};
\node (5) at (4.5, .45) {$s_1$};
\node (9) at (4.5,-2+0.3) {$t_1$};
\node (10) at (3.5,-2+0.3) {$t_2$};
\node (9) at (2.1,-2+0.3) {$\cdots$};
\node (9) at (0.5,-2+0.3) {$t_m$};
\end{tikzpicture}
\]
By definition, $s_1 < t_1$. 
Applying the triple condition yields that $s_2 < t_1$, thus $s_2 < t_2$.
Applying the triple condition again,
we see that $s_3 < t_2$, thus $s_3 < t_3$.
Repeating this argument,
we derive that $s_m < t_m$.
It, however, contradicts the assumption that $\sigma(j) > \sigma(i)$.
\end{proof}

\noindent
\begin{proof}[Proof of Proposition \ref{prop: row of 1 in source tab}]
Let $(j, \alpha_j) = \tau_0^{-1}(1)$
and $i = \sigma^{-1}(1)$.
Suppose on the contrary that $i \neq j$.
To begin with, we observe that $\alpha_{j} \geq 2$ since $\alpha_j$ has a removable node.

First, we suppose that $\alpha_i = 1$.
Let $t = (\tau_0)_{i,1}$.
Obviously, $t \neq 1$ and
$t$ does not lie to the immediate left of $t-1$.
Since $\sigma(i)=1$,
the box filled with $t-1$ cannot be placed on the first column,
and thus $t$ lies to  strictly left of $t-1$.
This contradicts the assumption that $\tau_0$ is a source tableau.

Second, we suppose that $\alpha_i > 1$ and $i <j$.
It follows from Lemma \ref{lem_removable1} that $\alpha_i \leq \alpha_j -2$.
Let $t = (\tau_0)_{i,\alpha_i}$ and $s = (\tau_0)_{j,\alpha_i+1}$.
By the triple condition, we have $t < s$.
Let $r$ be the leftmost one in the $j$th row among the entries less than $t$.
\[
\begin{tikzpicture}[scale=.55]
\draw[-] (0,0.3)--(8,0.3)--(8,1.3)--(0,1.3)--(0,0.3);
\draw[-] (2,0.3)--(2,1.3);
\draw[-] (3,0.3)--(3,1.3);
\draw[-] (4,0.3)--(4,1.3);
\draw[-] (5,0.3)--(5,1.3);
\draw[-] (6,0.3)--(6,1.3);
\draw[-] (7,0.3)--(7,1.3);
\draw[-] (8,0.3)--(8,1.3);
\draw[-] (0,2.5)--(3,2.5)--(3,3.5)--(0,3.5)--(0,2.5);
\draw[-] (2,2.5)--(2,3.5);
\node (0) at (-2, 3.1) {row $i$ :};
\node (1) at (-2, .5+0.4) {row $j$ :};
\node (2) at (2.45,2.1) {$\vdots$};
\node (4) at (3.5, .5+0.3) {$s$};
\node (5) at (4.55, .5+0.3) {$\cdots$};
\node (6) at (5.5, .5+0.3) {$r$};
\node (7) at (6.55, .5+0.3) {$\cdots$};
\node (8) at (7.5, .5+0.3) {$1$};
\node (4) at (2.5,3) {$t$};
\end{tikzpicture}
\]
Notice that the immediate left of $r$ cannot be $r+1$ by the choice of $r$.
Let $k$ be the index of row containing $t-1$. 
Since $\tau_0$ is a source tableau and $t$ does not lie to the immediate left of $t-1$,
it follows that $t-1$ lies to  weakly left of $t$.
Thus $i < k$ due to Lemma~\ref{lem: type dependence}.
Moreover, $\alpha_k \leq \alpha_i$. Otherwise, the triple condition breaks since $(i, \alpha_i+1) \not \in \tcd(\alpha)$.
As the next step, let $k'$ be the index of row containing $t-2$. 
In case where $t-1$ lies to the immediate left of $t-2$, we have that $i < k'$ and $\alpha_{k'} \leq \alpha_i$ since $k' = k$.
Otherwise, $t-2$ lies weakly left of $t-1$ since $\tau_0$ is a source tableau.
Using Lemma~\ref{lem: type dependence} and the triple condition again, we have that $i < k'$ and  $\alpha_{k'} \leq \alpha_i$.
Now let $k''$ be the index of row containing $r+1$.
Keeping on the above step, we derive that $i < k''$ and $\alpha_{k''} \leq \alpha_i$.
Obviously, $k'' \neq j$ and $r \not\in \Des(\tau_0)$, which contradicts the assumption that $\tau_0$ is a source tableau. 

Finally we suppose that $\alpha_i > 1$ and $i >j$.
Since $\alpha$ is compatible with $\sigma$ and $\alpha_j$ has a removable node,
we have that $\alpha_i < \alpha_j$.
Let $t = (\tau_0)_{i,\alpha_i}$ and $s = (\tau_0)_{j,\alpha_i}$.
If $s < t$,
then Lemma \ref{lem: type dependence} implies that $i <j$, which is absurd.
Hence $s > t$.
Let $r$ be the leftmost one in the $j$th row among the entries less than $t$.
Let $k''$ be the index of row containing $r+1$.
In the same manner as in case where $\alpha_i > 1$ and $i<j$, we derive that
$k'' \neq j$ and $\alpha_{k''} \leq \alpha_i$.
This yields $r \not\in \Des(\tau_0)$ in contradiction to the assumption that $\tau_0$ is a source tableau.
\end{proof}

We now state the classification of SPCT-cyclic $H_n(0)$-modules.

\begin{theorem}\label{thm: SPCT-cyclic}
  Let $\alpha$ be a composition of $n$.
\begin{enumerate}[label = {\rm (\alph*)}]
\item For every $\sigma \in \SG_{\ell(\alpha)}$,
  $\SPCTsa$ contains a unique source tableau
  if and only if $\alpha$ is $\sigma$-simple.
\item For every $\sigma \in \SG_{\ell(\alpha)}$, $\bfS^{\sigma}_{\alpha}$ is SPCT-cyclic 
  if and only if $\alpha$ is $\sigma$-simple.
\end{enumerate}
\end{theorem}

\begin{proof}
(a) To prove the ``only if'' part, assume that $\alpha$ is not $\sigma$-simple.
We will construct a source tableau $\tau_0$ in $\SPCTsa$ other than the canonical source tableau $\tauC$ of shape $\alpha$ and type $\sigma$.
From the assumption we have a PACD pair $(i,j)$ attached to $(\alpha \, ; \sigma)$
which satisfies
none of the conditions {\bf C1} and {\bf C2} in Definition \ref{def: simple}(b).
Equivalently, there is no $i < k < j$ such that $\sigma(i) < \sigma(k) < \sigma(j)$ and $\alpha_k = \alpha_j -1$,
and no $k > j$ such that $\sigma(i) < \sigma(k) < \sigma(j)$ and $\alpha_k = \alpha_j$.
Let us fix these $i$ and $j$ in the below.

In order to construct $\tau_0$, we list necessary notations in advance.
For better understanding see Example \ref{exam: proof of prop}.
Let
\begin{align*}
\mathfrak{I} := \{ 1 \leq k \leq \ell(\alpha) \mid \sigma(i) \leq \sigma(k) \leq \sigma(j) \} 
= \{k_1 < k_2 < \cdots < k_m\} \, .
\end{align*}
Let $\tau_{\mathfrak{I}}$ (resp., $\tau_{\mathfrak{I}^c}$) be the subfilling of the canonical source tableau $\tauC$ in $\SPCTsa$ consisting of the $k$th rows for all $k \in \mathfrak{I}$ (resp., $k \not\in \mathfrak{I}$).
Set $\beta := (\alpha_{k_1},\alpha_{k_2}, \ldots , \alpha_{k_m})$
and let $\sigma_\beta$ be the permutation given by the standardization of  $\sigma(k_1) \sigma(k_2) \cdots  \sigma(k_m)$.
Let $\kappaC$ be the canonical source tableau in $\rm{SPCT}^{\sigma_\beta}(\beta)$.
Notice that one can view $\tau_{\mathfrak{I}}$ as $\kappaC$ by substracting
$\sum_{ s < \sigma(i)} \alpha_{\sigma^{-1}(s)}$ from each entry in $\tau_{\mathfrak{I}}$.
Since  $\alpha_{\sigma^{-1}(j)}$ has a removable node in $\beta$ with respect to $\sigma_\beta$,
we can construct another source tableau $\hkappaC$ in $\rm{SPCT}^{\sigma_\beta}(\beta)$ due to Lemma \ref{lem : construct source}.
Finally, let $\widehat{\tau}_{\mathfrak{I}}$ be the filling obtained from $\hkappaC$ by adding $\sum_{ s < \sigma(i)} \alpha_{\sigma^{-1}(s)}$ to each entry in $\hkappaC$.

With the above notations, we define $\tau_0$ to be the filling obtained from $\tauC$ by substituting the $k_p$th row with the $p$th row of $\widehat{\tau}_{\mathfrak{I}}$ for $1 \le p \le m$.
Notice that each entry in $\tau_{\mathfrak{I}^c}$ is either strictly less than or strictly greater than all the entries in $\widehat{\tau}_{\mathfrak{I}}$, so $\tau_0$ is an SPCT.
In addition, both $\tauC$ and $\widehat{\tau}_{\mathfrak{I}}$ satisfy the condition for source tableaux.
Putting these together shows that $\tau_0$ is also a source tableau in $\SPCTsa$.

Next, let us prove the ``if'' part by induction on $|\alpha|$.
Our assertion is clear when $\alpha = (1)$.
Assume that our assertion holds for all compositions of $d-1 \ge 1$.
Let $\alpha$ be a composition of $d$ and
$\sigma$  any permutation in $\SG_{\ell(\alpha)}$ such that $\alpha$ is $\sigma$-simple.
By Proposition \ref{prop: row of 1 in source tab},
we have that if $\alpha$ is $\sigma$-simple and $\tau_0$ is a source tableau,
then $\tau_0^{-1}(1)$ lies on the rightmost of $\sigma^{-1}(1)$th row.
Let $\widehat{\tau}_0$ be the SPCT obtained by removing the box filled with $1$ from $\tau_0$ and then by
subtracting $1$ from each remaining entry.
It is obvious that
$\widehat{\tau}_0$ is a source tableau in $\rm{SPCT}_{\widehat{\alpha}}^{\boldsymbol{\sigma}}$ where
$\widehat{\alpha} := (\alpha_1, \cdots, \alpha_{\sigma^{-1}(1)}-1, \cdots, \alpha_{\ell(\alpha)})$ and
 \[\boldsymbol{\sigma} =
   \begin{cases}
    \sigma & \text{ if } \alpha_{\sigma^{-1}(1)} > 1 \\
    \sigma_{\downarrow} & \text{ if } \alpha_{\sigma^{-1}(1)} = 1
   \end{cases} \, .
   \]
Since $\widehat{\alpha}$ is a $\boldsymbol{\sigma}$-simple composition due to Lemma \ref{lem_removable2},
our induction hypothesis
implies that
$\widehat{\tau}_0$ is the canonical source tableau of shape $\widehat{\alpha}$ and type $\boldsymbol{\sigma}$.
Consequently, $\tau_0$ is the canonical source tableau of shape ${\alpha}$ and type $\sigma$
since the box $(\sigma^{-1}(1), \alpha_{\sigma^{-1}(1)})$ is filled with $1$, as required.

(b) 
The assertion can be obtained by putting Lemma~\ref{lem: unique source and sink}, \eqref{eq: decomp by eq rel}, and (a) together.
\end{proof}

\begin{example}\label{exam: proof of prop}
Let $\alpha = (2,4,4,2,3)$ and $\sigma = 24315$.
Then $\alpha$ is not $\sigma$-simple because the PACD pair $(3,5)$  satisfies neither {\bf C1} nor {\bf C2}.
Since $\sigma(3) = 3$, $\sigma(2) = 4$ and $\sigma(5) = 5$,
we have that $\mathfrak{I} = \{ 2, 3, 5 \}$.
In addition, $\beta = (\alpha_2, \alpha_3, \alpha_5) = (4,4,3)$ and $\sigma_\beta = 213$, the standardization of $\sigma(2) \sigma(3) \sigma(5)$.
Then we have
\[
 \tauC \, = \, {\small 
 \begin{array}{l}
 \begin{ytableau}
  4 & 3\\
  *(black!10) 12 & *(black!10)11 & *(black!10)10 & *(black!10)9 \\
  *(black!10)8 & *(black!10)7 & *(black!10)6 & *(black!10)5\\
  2 & 1 \\
  *(black!10)15 & *(black!10)14 & *(black!10)13
\end{ytableau}
\end{array}} \quad \text{and} \quad \
 \tau_0 \, = \, {\small 
 \begin{array}{l}\begin{ytableau}
  4 & 3\\
  *(black!10) 13 &*(black!10) 12 & *(black!10)11 & *(black!10)10   \\
  *(black!10)9 &*(black!10)8 & *(black!10)7 & *(black!10)6 \\
  2 & 1 \\
  *(black!10)15 & *(black!10)14 & *(black!10)5
\end{ytableau}
\end{array}}\ .
\]
In the left tableau, $\tau_\mathfrak{I}$ is the shaded part and $\tau_{\mathfrak{I}^c}$ is the unshaded part.
In the right tableau, $\widehat{\tau}_\mathfrak{I}$ is the shaded part.
\end{example}

When $\sigma$ is the longest element, Theorem~\ref{thm: SPCT-cyclic} can be stated in the following simple form. 

\begin{corollary}\label{Coro318}
Let $w_0$ be the longest element in $\SG_{\ell}$.
For a composition $\alpha$ of length $\ell$,
$\SPCT_{\alpha}^{w_0}$ is nonempty if and only if $\alpha$ is a partition.
Furthermore, $\bfS^{w_0}_{\alpha}$ is an SPCT-cyclic $H_n(0)$-module if and only if $\alpha$ is a partition of $n$.
\end{corollary}
\begin{proof}
The first assertion follows from Proposition \ref{Prop: comparibility}.
Notice that if $\alpha$ is a partition, then
$\alpha$ is $w_0$-simple since there are no PACD pairs attached to $(\alpha \, ; w_0)$.
Hence the second assertion follows from Theorem \ref{thm: SPCT-cyclic}.
\end{proof}

\section{Characteristic relations and the quasi-Schur expansion of $\ch([\bfSsa])$}\label{sec: Chac ral and quasi Schur exp}

From now on,  
we conventionally drop the superscript $\sigma$ from 
$\SPCTsa$, $\bfSsa$, $\bfS^{\sigma}_{\alpha,E}$, and $\mathcal{E}^{\sigma}(\alpha)$
when $\sigma = \id$.
It was shown in~\cite{15TW} that
\begin{align*}
\ch([\bfS_{\alpha}]) = \sum_{\tau \in \SPCT(\alpha)} F_{\comp(\tau)} = \calS_\alpha,
\end{align*}
where $\calS_\alpha$ is the \emph{quasisymmetric Schur function} introduced in~\cite{11HLMW}. However, in case where $\sigma \neq \id$, $\ch([\bfS_{\alpha}^\sigma])$ has not been well studied yet. 
In this section, we provide a recursive relation for $\ch([\bfS_{\alpha}^\sigma])$'s. Using this, we expand $\ch([\bfS_{\alpha}^\sigma])$ in terms of quasisymmetric Schur functions.

We begin with recalling that $H_n(0)$ acts on the right on $\C[\Z^n]$, the $\C$-span of $\Z^n$, by
\begin{align}\label{eq: bubble sorting operator}
\alpha \bubact \pi_i = \begin{cases}
\alpha \cdot s_i & \text{if $\alpha_i < \alpha_{i+1}$,}\\
\alpha & \text{otherwise}
\end{cases}
\end{align}
for $1\le i \le n-1$ and $\alpha = (\alpha_1,\ldots,\alpha_n) \in \Z^n$. Here $\alpha \cdot s_i$ is obtained from $\alpha$ by swapping $\alpha_i$ and $\alpha_{i+1}$ (see \cite{16Huang}).
Obviously, this action restricts to the $\C$-span of compositions of length $n$.

Let $\alpha$ be a composition and $w_0$ the longest element in $\SG_{\ell(\alpha)}$. For every $\sigma \in \SG_{\ell(\alpha)}$ and $\tab \in \PCT^{\sigma}(\alpha)$, define $\brho_\sigma(\tab)$ to be the filling obtained by
considering the entries of the $i$th column of $\tab$ in decreasing order and putting them in the $i$th column  of $\tal$ from top to bottom for all $1 \le i \le \alpha_{\max}$. 

\begin{theorem}\cite[Theorem 2.4]{19TW}
Let $\lambda$ be a partition and $\sigma \in \SG_{\ell(\lambda)}$. The map
\[
\brho_\sigma : \coprod_{\tal = \lambda} \PCT^\sigma(\alpha) \ra \PCT^{w_0}(\lambda), \quad \tab \mapsto \brho_\sigma(\tab)
\]
is a bijection.
\end{theorem}

Let $\bphi_{\sigma}$ denote the inverse of $\brho_\sigma$, which can be obtained by the following algorithm (see~\cite{19TW}).
\begin{algorithm}\label{alg: phi map} Let $\tab \in \PCT^{w_0}(\tal)$.
\begin{enumerate}[label = {\it Step \arabic*.}]
\item Consider the entries in the first column of $\tab$ and write them in rows $1,2,\ldots, \ell(\tal)$ so that the standardization of the word obtained by reading from top to bottom is $\sigma$.
\item Consider the entries in the second column in decreasing order and place each of them in the row with the smallest index so that the box to the immediate left of the number being placed is filled and the row entries weakly decrease when read from left to right.
\item Repeat {\it Step} 2 with the set of entries in the $k$th column for $k=3,4, ..., \lambda_1$.
\item Let $\bphi_{\sigma}(\tau)$ be the resulting tableau and terminate the algorithm.
\end{enumerate}
\end{algorithm}

Let $\lambda$ be a partition and $\sigma_1, \sigma_2 \in \SG_{\ell(\lambda)}$. Consider the map defined by
\begin{align*}
\mpsi_{\sigma_1}^{\sigma_2} : \coprod_{\tal = \lambda} \SPCT^{\sigma_1}(\alpha) \ra \coprod_{\tal = \lambda} \SPCT^{\sigma_2}(\alpha),
\quad \tab \mapsto \left(\bphi_{\sigma_2} \circ \brho_{\sigma_1}\right)(\tab).
\end{align*}

\begin{example}
Let
\[
\tab = 
\begin{array}{l}
\begin{ytableau}
5 & 2 \\
7 & 6 & 4 \\
3 & 1
\end{ytableau}
\end{array}
\in \SPCT^{231}(2,3,2).
\]
Then we have
\[
\brho_{231}(\tab) = 
\begin{array}{l}
\begin{ytableau}
7 & 6 & 4 \\
5 & 2 \\
3 & 1
\end{ytableau}
\end{array}
\in \SPCT^{321}(3,2,2) \quad \text{and} \quad 
\mpsi_{231}^{123}(\tab) = 
\begin{array}{l}
\begin{ytableau}
3 & 2 \\
5 & 1 \\
7 & 6 & 4
\end{ytableau}
\end{array}
\in \SPCT(2,2,3).
\]
\end{example}

\begin{remark}\label{rem: descent preserve}
Let $\lambda$ be a partition and $\sigma \in \SG_{\ell(\lambda)}$. Since $\bphi_{\sigma}$ and $\brho_{\sigma}$ preserve the entries in each column, they preserve $\Des(\tau)$ for every $\tau \in \coprod_{\tal = \lambda} \SPCT^{\sigma}(\alpha)$. Thus, for any $\sigma_1, \sigma_2 \in \SG_{\ell(\lambda)}$, $\mpsi_{\sigma_1}^{\sigma_2}$ also preserves $\Des(\tau)$.
\end{remark}

The following lemma is indispensable in proving Proposition~\ref{prop: psi map image}
which leads us to the main result of this section.

\begin{lemma}\label{lem: T_{i',k} all equal}
Given a composition $\alpha$ and $\sigma \in \SG_{\ell(\alpha)}$, 
let $\tab \in \SPCTsa$. For $1 \le i \le \ell(\alpha)$, set $\htau := \mpsi^{\sigma s_i}_{\sigma}(\tab)$ and  $p:= \min\{\alpha_i,\alpha_{i+1}\}$.

\begin{enumerate}[label = {\rm(\alph*)}]
\item If there is $1 \le k \le p$ such that $\tab_{i,k} > \tab_{i+1,k}$, then $\tab_{i,m} > \tab_{i+1,m}$ for all $k \le m \le p$.
\item If the $k$th column of $\htau$ is equal to that of $\tab$ for some $k \ge 1$, then the $(k+1)$st column of $\htau$ is also equal to that of $\tab$.
\item For $1 \le k \le p - 1$, suppose that $\tab_{i,k} > \tab_{i+1,k}$ and the $k$th column of $\htau$ is obtained from the $k$th column of $\tab$ by swapping $\tab_{i,k}$ and  $\tab_{i+1,k}$.
\begin{enumerate}[label = {\rm (\roman*)}]
\item If $\tab_{i+1,k} >  \tab_{i,k+1}$, then the $(k+1)$st column of $\htau$ is equal to that of $\tab$.
\item If $\tab_{i+1,k} < \tab_{i,k+1}$, then the $(k+1)$st column of $\htau$ is obtained from the $(k+1)$st column of $\tab$ by swapping $\tab_{i,k+1}$ and $\tab_{i+1,k+1}$.
\end{enumerate}
\item For $1 \le k \le p - 1$, suppose that $\tab_{i,k} < \tab_{i+1,k}$ and the $k$th column of $\htau$ is obtained from the $k$th column of $\tab$ by swapping $\tab_{i,k}$ and  $\tab_{i+1,k}$.
\begin{enumerate}[label = {\rm (\roman*)}]
\item If $\tab_{i,k} > \tab_{i+1,k+1}$, then the $(k+1)$st column of $\htau$ is equal to that of $\tab$.
\item If $\tab_{i,k} < \tab_{i+1,k+1}$, then the $(k+1)$st column of $\htau$ is obtained from the $(k+1)$st column of $\tab$ by swapping $\tab_{i,k+1}$ and  $\tab_{i+1,k+1}$.
\end{enumerate}
\end{enumerate}
\end{lemma}

\begin{proof}
For simplicity, set $\uptau := \brho_{\sigma}(\tab)$. 
Note that 
\begin{align}\label{redefine tau and tau-hat}
\tab = \mpsi_{\sigma}^\sigma(\tab) = \bphi_{\sigma}(\uptau) \quad \text{and} \quad \htau = \mpsi_{\sigma}^{\sigma s_i}(\tab) = \bphi_{\sigma s_i}(\uptau).
\end{align}

(a)
Since the $i$th row is decreasing, by the triple condition, we have the inequality $\tab_{i,k} > \tab_{i,k+1} > \tab_{i+1,k+1}$.
Our assertion can be easily verified in an inductive manner.

(b) 
From Algorithm~\ref{alg: phi map} one can see that the $(k+1)$st column of $\bphi_{\sigma}(\uptau)$ and $\bphi_{\sigma s_i}(\uptau)$ are completely determined by the $k$th column of $\bphi_{\sigma}(\uptau)$ and $\bphi_{\sigma s_i}(\uptau)$, respectively.
Thus the assertion is straightforward from the assumption.

(c)
Due to (a), we have $\tab_{i,k+1} > \tab_{i+1,k+1}$. 
To begin with, consider the entries $\tab_{j,k+1}$ greater than $\tab_{i,k+1}$. 
Since the $k$th column of $\tab$ is equal to the $k$th column of $\htau$ except for the $i$th and $(i+1)$st entries, by Algorithm~\ref{alg: phi map} again, we have  
\begin{align}\label{eq: htau eq tau j>i}
\htau_{j,k+1} \underset{\text{by }\eqref{redefine tau and tau-hat}}{=}  \bphi_{\sigma s_i}(\uptau)_{j,k+1} \underset{\text{by Algorithm~\ref{alg: phi map}}}{=}  \bphi_{\sigma}(\uptau)_{j,k+1} \underset{\text{by }\eqref{redefine tau and tau-hat}}{=}  \tab_{j,k+1}.
\end{align}

Now let us prove the assertion (i). We first show that $\htau_{i,k+1} = \tab_{i,k+1}$.
Since $\tab_{i,k+1} > \tab_{i+1,k+1}$,
$\tab_{i,k+1}$ should be placed before $\tab_{i+1,k+1}$ in the $(k+1)$st column of $\htau$ when applying $\mpsi^{\sigma s_i}_{\sigma}$ to $\tab$.
Combining~\eqref{eq: htau eq tau j>i} and the assumption $\htau_{i,k} = \tab_{i+1,k} > \tab_{i,k+1}$, we deduce that $\tab_{i,k+1}$ appears weakly above $(i,k+1)$ in $\htau$. 
On the other hand, under Algorithm~\ref{alg: phi map}, $\tab_{i,k+1}$ should be placed at the same position in $\tab$ and $\htau$ because of \eqref{eq: htau eq tau j>i} together with the assumption that the upper $i-1$ entries of the $k$th column of $\tab$ are equal to those in the $k$th column of $\htau$. Thus $\htau_{i,k+1} = \tab_{i,k+1}.$

To see that $\htau_{i+1,k+1} = \tab_{i+1,k+1}$, let $\htau_{i+1,k+1} = \tab_{r,k+1} < \tab_{i,k+1}$.
We have two cases.
\smallskip

{\it Case 1: $\tab_{r,k+1} < \tab_{i+1,k+1}< \tab_{i,k+1}$.}
First of all, note that $\tab_{i+1,k+1}$ is placed before $\tab_{r,k+1}$ in the $(k+1)$st column of $\htau$ when applying $\mpsi^{\sigma s_i}_{\sigma}$ to $\tab$.
In the above, we have placed all entries $\tab_{j,k+1}$ with $\tab_{j,k+1}\ge \tab_{i,k+1}$ in such a way that $\htau_{j,k+1} = \tab_{j,k+1}$. 
We now place all entries $\tab_{j,k+1}$ such that $\tab_{i+1,k+1} \le \tab_{j,k+1} < \tab_{i,k+1}$. Clearly they cannot be placed at the positions $(i,k+1)$ and $(i+1, k+1)$ since 
the former is already filled with $\tab_{i,k+1}$
and the latter is subject to be filled with $\tab_{r,k+1}$.
Since the $k$th column of $\tab$ is equal to the $k$th column of $\htau$ except for the $i$th and $(i+1)$st entries, by Algorithm~\ref{alg: phi map}, we see inductively that $\tab_{j,k+1}=\htau_{j,k+1}$ for all these entries.
But this cannot occur since $\tab_{i+1,k+1} \ne \tau_{r,k+1}$.
\smallskip

{\it Case 2: $\tab_{i+1,k+1} \le \tab_{r,k+1}< \tab_{i,k+1}$.}
First, we claim that $r \le i+1$. Otherwise, the triple $(\tab_{i+1,k},\tab_{i+1,k+1},\tab_{r,k+1})$
does not satisfy the triple condition in $\tab$ because 
$$\tab_{i+1,k}>\tab_{i,k+1} >\tab_{r,k+1} \quad \text{and} \quad \tab_{i+1,k+1} \le \tab_{r,k+1}.$$
As in {\it Case 1}, all entries $\tab_{j,k+1}$ such that $\tab_{r,k+1} < \tab_{j,k+1} < \tab_{i,k+1}$ 
are placed in the $(k+1)$st column of $\htau$ in such a way that $\htau_{j,k+1} = \tab_{j,k+1}$.  
Now let us place $\tab_{r,k+1}$.
If $r < i+1$, then the box $(r,k+1)$ is unfilled in $\htau$ and therefore $\tab_{r,k+1}$ should be placed in this box by Algorithm~\ref{alg: phi map}. This contradicts the assumption $\htau_{i+1,k+1} = \tab_{r,k+1}$.  
As a consequence, we deduce that $r=i+1$.
\smallskip

Up to now, we have placed all entries $\tab_{j,k+1}$ with $\tab_{j,k+1}\ge \tab_{i+1,k+1}$ in such a way that $\htau_{j,k+1} = \tab_{j,k+1}$.
Since the $k$th column of $\tab$ is equal to the $k$th column of $\htau$ except for the $i$th and $(i+1)$st entries, by Algorithm~\ref{alg: phi map}, the remaining entries should also be filled in such a way that $\htau_{j,k+1} = \tab_{j,k+1}$. 
So, we are done.

Next, let us prove the second assertion (ii).
Note that  
\begin{align}\label{eq: tau and htau eq for case2}
\htau_{i,k} = \tab_{i+1,k} < \tab_{i,k+1} \quad \text{and} \quad \htau_{i+1,k} = \tab_{i,k} > \tab_{i,k+1}.
\end{align}
By~\eqref{eq: htau eq tau j>i} and \eqref{eq: tau and htau eq for case2}, the entry $\tab_{i,k+1}$ should be placed weakly above $(i+1,k+1)$ in the $(k+1)$st column of $\htau$ except for $(i,k+1)$. Using the same argument as in the above, one sees that $\htau_{i+1,k+1} = \tab_{i,k+1}$.
Ignoring the row containing $\tab_{i,k}$, one may assume that the $k$th column of $\htau$ is equal to that of $\tab$, so   
the remaining entries should also be filled in such a way that $\htau_{i,k+1} = \tab_{i+1,k+1}$ and $\htau_{j,k+1} = \tab_{j,k+1}$ for $j \neq i,i+1$.

(d) The assertion can be proved in a similar way as in (c).
\end{proof}

\begin{proposition}\label{prop: psi map image}
Let $\alpha$ be a composition and $\sigma \in \SG_{\ell(\alpha)}$.
If there is $1 \le i \le \ell(\alpha)-1$ such that $\ell(\sigma s_i)  <  \ell(\sigma)$, then 
\begin{align*}
\mpsi^{\sigma s_i}_{\sigma}(\SPCTsa) = \coprod_{\alpha = \beta \bubact \pi_{i}} \SPCT^{\sigma s_i}(\beta).
\end{align*}
\end{proposition}

\begin{proof}
First, we deal with the case where $\alpha_i  < \alpha_{i+1}$.
The assumption $\ell(\sigma s_i)  <  \ell(\sigma)$ implies that $\sigma(i) > \sigma(i+1)$. Therefore $\SPCTsa$ is the empty set by Proposition~\ref{Prop: comparibility}.
On the other hand, one can easily see that there are no $\beta \models n$ such that $\alpha = \beta \bubact \pi_i$, so the right hand side is also the empty set.

From now on, we assume that $\alpha_i  \ge \alpha_{i+1}$.
Let us prove the inclusion $\subseteq$.
Let $\tab \in \SPCTsa$ and set $\htau := \mpsi^{\sigma s_i}_{\sigma}(\tab)$.
We have two cases.
\smallskip

{\it Case 1: There is $1 \le k \le \alpha_{i+1} -1$ such that $\tab_{i+1,k} > \tab_{i,k+1}$.}
Let $1 \le k_0 \le \alpha_{i+1} - 1$ be the smallest integer such that $\tab_{i+1,k_0} > \tab_{i,k_0 + 1}$.
Then, by Lemma~\ref{lem: T_{i',k} all equal}(b) and Lemma~\ref{lem: T_{i',k} all equal}(c), we have the following identities:
\begin{align*}
&\htau_{i,m} = \tab_{i+1,m} \quad \text{and} \quad
\htau_{i+1,m} = \tab_{i,m} \quad \text{for $1 \le m \le k_0$},\\
&\htau_{i,m} = \tab_{i,m} \quad \text{and} \quad
\htau_{i+1,m} = \tab_{i+1,m} \quad \text{for $k_0+1 \le m \le \alpha_{i+1}$},\\
&\htau_{j,m}  = \tab_{j,m} \quad \text{for $j \neq i,i+1$ and $1 \le m \le \alpha_{i+1}$}.
\end{align*}
By Lemma~\ref{lem: T_{i',k} all equal}(b), 
$\htau_{j,m}  = \tab_{j,m}$ for $1 \le j \le \ell(\alpha)$ and $m \ge \alpha_{i+1}+1$,
and thus the shape of $\htau$ is $\alpha$. Since $\alpha = \alpha \bubact \pi_i$, we are done.
\smallskip

{\it Case 2: $\tab_{i+1,k} < \tab_{i,k+1}$ for all $1 \le k \le \alpha_{i+1} -1$.}
By Lemma~\ref{lem: T_{i',k} all equal}(c), we have the following identities:
\begin{align*}
&\htau_{i,m} = \tab_{i+1,m} \quad \text{and} \quad \htau_{i+1,m} = \tab_{i,m} \quad \text{for $1\le m \le \alpha_{i+1}$},\\
&\htau_{j,m}  = \tab_{j,m} \quad \text{for $j \neq i,i+1$ and $1 \le m \le \alpha_{i+1}$}.
\end{align*}
First, consider the case where  $\tab_{i+1,\alpha_{i+1}} > \tab_{i,\alpha_{i+1}+1}$.
Since $\htau_{i,\alpha_{i+1}} = \tab_{i+1,\alpha_{i+1}} > \tab_{i,\alpha_{i+1}+1}$, the entry $\tau_{i, \alpha_{i+1}+1}$ should be placed on $(i,\alpha_{i+1}+1)$ in $\htau$ when applying $\mpsi^{\sigma s_i}_{\sigma}$ to $\tab$. By Lemma~\ref{lem: T_{i',k} all equal}(b), 
$\htau_{j,m}  = \tab_{j,m}$ for $1 \le j \le \ell(\alpha)$ and $m \ge \alpha_{i+1}+1$,
thus the shape of $\htau$ is $\alpha$. Since $\alpha = \alpha \bubact \pi_i$, we are done.\\
Next, consider the case where $\tab_{i+1,\alpha_{i+1}} < \tab_{i,\alpha_{i+1}+1}$.
Since $\htau_{i,\alpha_{i+1}} = \tab_{i+1,\alpha_{i+1}} < \tab_{i,\alpha_{i+1}+1}$ and $\htau_{i+1,\alpha_{i+1}} = \tab_{i,\alpha_{i+1}} > \tab_{i,\alpha_{i+1}+1}$, $\tau_{i, \alpha_{i+1}+1}$ should be placed on $(i+1, \alpha_{i+1}+1)$ in $\htau$  when applying $\mpsi^{\sigma s_i}_{\sigma}$ to $\tab$.
By Lemma~\ref{lem: T_{i',k} all equal}(c), we have
\begin{align*}
&\htau_{i+1,m} = \tab_{i,m} \quad \text{for $m \ge \alpha_{i+1} + 1$},\\
&\htau_{j,m}  = \tab_{j,m} \quad \text{for $j \neq i, i+1$ and $m \ge \alpha_{i+1}+1$},
\end{align*}
thus the shape of $\htau$ is $\alpha \cdot s_i$. 
Here we are assuming that $\tab_{j,m} = 0$ for $(j,m) \not\in \tcd(\alpha)$.
Since $\alpha = (\alpha \cdot s_i) \bubact \pi_i$, we are done.
\medskip

Hence we verified the inclusion $\subseteq$.

Next, let us prove the opposite inclusion $\supseteq$, equivalently,
\begin{align*}
\SPCTsa \supseteq \mpsi_{\sigma s_i}^{\sigma}\left( \coprod_{\alpha = \beta \bubact \pi_{i}} \SPCT^{\sigma s_i}(\beta) \right).
\end{align*}
Let $\tab \in \SPCT^{\sigma s_i}(\beta)$ for some $\beta$ with $\alpha = \beta \bubact \pi_{i}$ and set $\otau := \mpsi_{\sigma s_i}^{\sigma}(\tab)$. Since $\ell(\sigma s_i)  <  \ell(\sigma)$, $\sigma s_i(i) < \sigma s_i (i+1)$, and therefore $\tab_{i,1} < \tab_{i+1,1}$. We have two cases.
\smallskip

{\it Case 1: $\beta = \alpha \cdot s_i$.}
We claim that there is no $1\le k \le \beta_i$ such that $\tab_{i,k} > \tab_{i+1,k}$. 
Otherwise, $\tab_{i,k} > \tab_{i+1,k} > \tab_{i+1,k+1}$ by row-decreasing condition. By the triple condition, we have $\tab_{i,k+1} > \tab_{i+1,k+1}$.
One can see inductively that $\tab_{i,m} > \tab_{i+1,m}$ for all $k \le m \le \beta_i$.
Since $\tab_{i,\beta_i} > \tab_{i+1,\beta_i}$, the box $(i,\beta_i+1)$ is not empty in $\tab$ by the triple condition, which is absurd. 
Therefore $\tab_{i,m} < \tab_{i+1,m}$ for all $1 \le m \le \beta_i$.
By the definition of $\mpsi^{\sigma}_{\sigma s_i}$, we have
\begin{align*}
&\otau_{i,m} = \tab_{i+1,m} \quad \text{for $1 \le m \le \beta_{i+1}$}, \quad \otau_{i+1,m} = \tab_{i,m} \quad \text{for $1 \le m \le \beta_{i}$},\\
&\otau_{j,m}  = \tab_{j,m} \quad \text{for $j \neq i,i+1$ and $1 \le m \le \beta_{j}$}
\end{align*}
and hence $\otau \in \SPCT^{\sigma}(\alpha)$ since $\beta = \alpha \cdot s_i$.
\smallskip

{\it Case 2: $\beta = \alpha$.}
First, consider the case where $\tab_{i,k} > \tab_{i+1,k}$ for some $1 \le k \le \alpha_{i+1}$.
Let $1 \le k_0 \le \alpha_{i+1}$ be the smallest integer such that $\tab_{i,k_0} > \tab_{i+1,k_0}$.
For $1 \le m < k_0 - 1$, one can see that $\tab_{i+1, m+1} > \tab_{i,m+1}$ and thus $\tab_{i+1, m+1} > \tab_{i,m}$ by the triple condition.
And, for $k_0 \le m \le \alpha_{i+1}$, we have $\tab_{i,m} > \tab_{i+1,m}$ by Lemma~\ref{lem: T_{i',k} all equal}(a). 
In the same way as in {\it Case 1} in the proof of the inclusion $\subseteq$, using Lemma~\ref{lem: T_{i',k} all equal}(b) and Lemma~\ref{lem: T_{i',k} all equal}(d), we can derive that
\begin{align*}
&\otau_{i,m} = \tab_{i+1,m} \quad \text{and} \quad
\otau_{i+1,m} = \tab_{i,m} \quad \text{for $1 \le m \le k_0 - 1$},\\
&\otau_{i,m} = \tab_{i,m} \quad \text{for $k_0 \le m \le \alpha_i$},\\
&\otau_{i+1,m} = \tab_{i+1,m} \quad \text{for $k_0 \le m \le \alpha_{i+1}$},\\
&\otau_{j,m}  = \tab_{j,m} \quad \text{for $j \neq i,i+1$ and $1 \le m \le \alpha_j$.}
\end{align*}
Hence $\otau \in \SPCT^{\sigma}(\alpha)$.\\
Next, consider the case where $\tab_{i,m} < \tab_{i+1,m}$ for all $1 \le m \le \alpha_{i+1}$.
By the definition of $\mpsi^{\sigma}_{\sigma s_i}$, we have
\begin{align*}
&\otau_{i,m} = \tab_{i+1,m} \quad \text{for $1 \le m \le \alpha_{i}$}, \\
&\otau_{i+1,m} = \tab_{i,m} \quad \text{for $1 \le m \le \alpha_{i+1}$},\\
&\otau_{j,m}  = \tab_{j,m} \quad \text{for $j \neq i,i+1$ and $1 \le m \le \alpha_j$.}
\end{align*}
Hence $\otau \in \SPCT^{\sigma}(\alpha)$.
\medskip

This completes the proof.
\end{proof}

In \cite[Section 5]{15TW}, Tewari and van Willigenburg expanded $\ch([\bfS_\alpha])$ in terms of fundamental quasisymmetric functions.
In \cite[Section 6]{19TW}, they also expanded $\ch([\bigoplus_{\sigma \in \SG_{\ell(\alpha)}}\bfS_\alpha^\sigma])$ in terms of fundamental quasisymmetric functions. 
Following their method with slight modification, one can obtain the expansion of $\ch([\bfSsa])$ in terms of fundamental quasisymmetric functions 
\begin{align*}
\ch([\bfSsa]) = 
\sum_{\tau \in \SPCT^{\sigma}(\alpha)} F_{\comp(\tau)}.
\end{align*}

Now we are ready to state the main result of this section.
\begin{theorem}\label{thm: charcteristic of SPCT}
For $\alpha \models n$ and $\sigma \in \SG_{\ell(\alpha)}$,
if there is $1 \le i \le \ell(\alpha)-1$ such that $\ell(\sigma s_i) < \ell(\sigma)$, then 
\begin{align*}
\ch([\bfSsa]) = \sum_{\alpha = \beta \bubact \hpi_{i}} \ch([\bfS^{\sigma s_i}_\beta]).
\end{align*}
\end{theorem}

\begin{proof}
The set of all entries in each column is invariant under $\bphi_{\sigma s_i}$ and $\brho_{\sigma}$.
This implies that the descent set of an SPCT is invariant under $\bphi_{\sigma s_i}$ and $\brho_{\sigma}$,
so invariant under $\mpsi_{\sigma}^{\sigma s_i}$.
Thus, by Proposition~\ref{prop: psi map image}, we have
\begin{equation*}
\ch([\bfSsa]) 
= \sum_{\tau \in \SPCT^{\sigma}(\alpha)} F_{\comp(\tau)} 
= \sum_{\tau \in \coprod_{\alpha = \beta \bubact \pi_{i}}\SPCT^{\sigma s_i}(\beta)} F_{\comp(\tau)}
= \sum_{\alpha = \beta \bubact \hpi_{i}} \ch([\bfS^{\sigma s_i}_\beta]).\qedhere
\end{equation*}
\end{proof}

The following corollary is an immediate consequence of Theorem~\ref{thm: charcteristic of SPCT}.

\begin{corollary}\label{cor: qschur exp}
Let $\alpha$ be a composition. For any $\sigma \in \SG_{\ell(\alpha)}$, we have
\begin{align*}
\ch([\bfSsa]) &= \sum_{\alpha = \beta \bubact \hpi_{\sigma}} \calS_{\beta},
\end{align*}
where $\calS_{\beta}$ is the quasisymmetric Schur function.
Furthermore, if $\alpha$ is a partition and $w_0$ is the longest element in $\SG_{\ell(\alpha)}$, 
then we have 
\begin{align*}
\ch([\bfS_\alpha^{w_0}]) &= s_\alpha, 
\end{align*}
where $s_\alpha$ is the Schur function.
\end{corollary}

We close this section by providing a new $\mathbb{Z}$-basis of $\Qsym$ consisting of suitable $\ch([\bfSsa])$'s.
For a partition $\lambda$, we let $(\SG_{\ell(\lambda)})_{\lambda}$ be the stabilizer subgroup of $\SG_{\ell(\lambda)}$ with respect to $\lambda$, that is, $(\SG_{\ell(\lambda)})_{\lambda} = \{\sigma \in \SG_{\ell(\lambda)} \mid \lambda \cdot \sigma = \lambda\}$.  
Define $\mathbf{I}_\lambda$ to be the set of minimal length coset representatives of $\SG_{\ell(\lambda)}/ (\SG_{\ell(\lambda)})_{\lambda}$ and set
\[
\mathbf{B}_n := \bigcup_{\lambda \vdash n} \left\{ \ch([\bfS_\lambda^{\sigma}]) \; \middle| \; \sigma \in \mathbf{I}_\lambda \right\} \, .
\]

\begin{corollary}\label{cor: basis for Qsym}
With the above notation,
$\mathbf{B}_n$ is a $\mathbb{Z}$-basis of $\Qsym_n$.
\end{corollary}
\begin{proof}
Let $\lambda = (\lambda_1, \lambda_2, \ldots , \lambda_\ell) \vdash n$ and $\sigma \in \SG_\ell$.
We will first prove that $(\lambda \cdot \sigma^{-1})\bubact \hpi_{\sigma} = \lambda$ using induction on the length of $\sigma$.
If $\sigma = s_i$,
then  
\[
(\lambda \cdot s_i)\bubact \hpi_{i} = (\lambda_1, \ldots, \lambda_{i+1}, \lambda_i, \ldots, \lambda_\ell) \bubact \hpi_i = (\lambda_1, \ldots, \lambda_{i}, \lambda_{i+1}, \ldots, \lambda_\ell)
\]
since $\lambda_i \geq \lambda_{i+1}$.
For $1 \le d < \ell(w_0)$, assume that $(\lambda \cdot \gamma^{-1})\bubact \hpi_{\gamma} = \lambda$ for every $\gamma \in \SG_\ell$ with $\ell(\gamma) = d$.
For $\sigma \in \SG_\ell$ of length $d+1$,
one can write $\sigma$ as $s_i \gamma$ for some $1 \leq i < \ell$ and $\gamma \in \SG_\ell$ with $\ell(\gamma) = d$.
Since $\ell(s_i \gamma) > \ell(\gamma)$, that is, $\gamma^{-1}(i) < \gamma^{-1}(i+1)$, we see that $\lambda_{\gamma^{-1}(i)} \geq \lambda_{\gamma^{-1}(i+1)}$.
Hence we have
\begin{align*}
  (\lambda \cdot \sigma^{-1})\bubact \hpi_{\sigma}  & = (\lambda_{\gamma^{-1}(s_i(1))}, \ldots, \lambda_{\gamma^{-1}(s_i(i))}, \lambda_{\gamma^{-1}(s_i(i+1))}, \ldots, \lambda_{\gamma^{-1}(s_i(\ell))} )\bubact \hpi_{\sigma} \\
   &= \left( (\lambda_{\gamma^{-1}(1)}, \ldots, \lambda_{\gamma^{-1}(i+1)}, \lambda_{\gamma^{-1}(i)}, \ldots, \lambda_{\gamma^{-1}(\ell)} )\bubact \hpi_i \right)\bubact \hpi_{\gamma} \\
  &= (\lambda_{\gamma^{-1}(1)}, \ldots, \lambda_{\gamma^{-1}(i)}, \lambda_{\gamma^{-1}(i+1)}, \ldots, \lambda_{\gamma^{-1}(\ell)} )\bubact \hpi_{\gamma}\\
  & = \lambda \, ,
\end{align*}
as desired.

Applying Proposition~\ref{prop: psi map image} repeatedly,
one can see that the $\mpsi^{\rm{id}}_{\sigma}$ is a bijection between $\SPCTsa$ and $ \coprod_{\alpha = \beta \bubact \hpi_\sigma} \SPCT(\beta)$.
This shows that the quasi-Schur expansion of $\ch([\bfS_\lambda^{\sigma}])$ is multiplicity-free.
Thus we deduce from Corollary~\ref{cor: qschur exp} that
\[
\ch([\bfS_\lambda^{\sigma}]) = \calS_{\lambda \cdot \sigma^{-1}} + \sum_{\underset{\gamma \, \lneq  \, \sigma^{-1}}{\beta = \lambda \cdot \gamma}} \calS_\beta \, .
\]
Here $<$ denotes the weak  Bruhat  order on $\SG_{\ell}$.
Now our assertion follows from the fact: 
$\lambda \cdot \gamma^{-1}
 = \lambda \cdot \rho^{-1}$ if and only if $\gamma$ and $\rho$ are in the same left coset of $\SG_{\ell(\lambda)}/ (\SG_{\ell(\lambda)})_{\lambda}$.
This completes the proof.
\end{proof}

\section{The projective cover of the canonical submodule of $\bfSsa$}\label{Sect5}
To begin with, we recall the notion of a projective cover.
Let $R$ be a left artin ring and $A,B$ be finitely generated $R$-modules.
An epimorphism $f:A\to B$ is called an \emph{essential epimorphism} if a morphism $g: X\to A$ is an epimorphism 
whenever $f \circ g:X\to B$ is an epimorphism, or equivalently, $\ker f \subset {\rm rad}(A)$.
A \emph{projective cover} of $A$ is an essential epimorphism $f:P\to A$ with $P$ a projective $R$-module, 
which always exists and is unique up to isomorphism. 
It plays an extremely important role in understanding the structure of $A$ (see~\cite{95ARS}).

In this section, we will only consider $\alpha \models n$ and $\sigma \in \SG_{\ell(\alpha)}$ such that $\alpha$ is compatible with $\sigma$ since $\bfSsa$ is zero otherwise (see Proposition~\ref{Prop: comparibility}).
Recall that we have introduced the canonical source tableau $\tauC$ in Definition~\ref{def: canonical tableau}.
We call the class containing $\tauC$, denoted by $C$, the \emph{canonical class} in $\mathcal{E}^\sigma(\alpha)$, and ${\bfS}^\sigma_{\alpha,C}$ the \emph{canonical submodule} of $\bfSsa$.
The purpose of this section is to find the projective cover of $\bfSsaC$. 
More precisely, we show that ${\bfS}^\sigma_{\alpha,C}$ appears as a homomorphic image of a projective indecomposable module $\calP_I$ for some $I \subseteq [n-1]$.

Fayers~\cite{05Fayers} introduced an automorphism $\Th$ of $H_n(0)$ defined by
\begin{align*}
\Th: \hpi_i \mapsto -\opi_i = 1 - \hpi_i \quad \text{for } 1 \leq i \leq n-1.
\end{align*}
Given an $H_n(0)$-module $M$,   
it gives another $H_n(0)$-action $\Thact$ on the vector space $M$ defined by 
$$
\hpi_i \Thact v := \Th(\hpi_i) \cdot v = -\opi_i \cdot v \quad \text{for } 1 \leq i \leq n-1.
$$
The resulting $H_n(0)$-module is denoted by $\theta[M]$.
In particular, for the $H_n(0)$-module $\bfP_\alpha$ in Subsection~\ref{subsec: PIM}, 
the $H_n(0)$-action on $\theta[\bfP_\alpha]$ is defined by
\begin{align*}
\hpi_i * T = \begin{cases}
T, & \text{if $i$ is in a higher row of $T$ than $i+1$},\\
0, & \text{if $i$ is in the same row of $T$ as $i+1$},\\
- s_i \cdot T, & \text{if $i$ is in a lower row of $T$ than $i+1$}
\end{cases}
\end{align*}
for $i \in [n-1]$ and $T \in \SRT(\alpha)$. 
It was remarked in~\cite[Remark 5.2]{16Huang} that there is another $H_n(0)$-action on $\bfP_\alpha$ defined as follows:
for $i \in [n-1]$ and $T \in \SRT(\alpha)$,
\begin{align*}
\hpi_i \star T: = \begin{cases}
T, & \text{if $i$ is in a higher row of $T$ than $i+1$},\\
0, & \text{if $i$ is in the same row of $T$ as $i+1$},\\
s_i \cdot T, & \text{if $i$ is in a lower row of $T$ than $i+1$.}
\end{cases}
\end{align*}
The resulting $H_n(0)$-module is denoted by $\overline{\hThbfP}$. Therein it was stated that $\theta[\bfP_\alpha]$ is isomorphic to $\overline{\hThbfP}$ without an explicit isomorphism. For the completeness, let us briefly explain how to construct an isomorphism. 

Let $T_0 \in \SRT(\alpha)$ be the standard ribbon tableau obtained by filling $\trd(\alpha)$ 
with the entries $1,2,\ldots,n$ from top to bottom and from left to right. 
Note that $\bfP_\alpha$ is generated by $T_0$ cyclically (see~\cite{16Huang}).
For this reason, we call $T_0$ the \emph{source tableau of $\bfP_\alpha$}.
Let $T \in \SRT(\alpha)$.
One can easily see that there exists a unique minimal length coset representative $\rho_T \in \SG_n / \SG_{n,(\Des(\alpha))^{\rm c}}$ with $T = \pi_{\rho_T} * T_0$
and the map 
\begin{equation}\label{Eq:Isom1}
\iota : \ThbfP \to \overline{\hThbfP},  \qquad T \mapsto {\rm sgn}(\rho_T) T
\end{equation}
is an $H_n(0)$-module isomorphism. 

For $T \in \SRT(\alpha)$ and $\sigma \in \SG_{\ell(\alpha^{\rm c})}$, 
we define $\tab_T$ by the filling of $\tcd(\alpha^{\rm c} \cdot \sigma)$ whose
$i$th row is filled with the entries in the $\sigma(i)$th column of $T$ in decreasing order for $i = 1, 2, \ldots, \ell(\alpha^{\rm c})$.

Consider the $\C$-linear map
$$
\Phi^\sigma_{\alpha}: \overline{\hThbfP} \to \bfS^{\sigma}_{\alpha^{\rm c} \cdot \sigma}
$$ 
defined by 
\begin{displaymath}
\Phi^\sigma_{\alpha}(T) =
\begin{cases}
\tab_T & \text{if it is an SPCT of type $\sigma$},\\
0 & \text{otherwise}.
\end{cases}
\end{displaymath}

\begin{example}\label{Example52}
Let $\alpha = (2,2,1,1,1,2,1)$. 
If $\sigma = 2314 \in \SG_4$, then $\alpha^{\rm c} \cdot \sigma = (2,5,1,2)$
since $\alpha^{\rm c} = (1,2,5,2)$.
Let 
\begin{displaymath}
T_1 = \begin{array}{c}
\begin{ytableau}
\none & \none & \none & 9  \\
\none & \none & 4  & 10 \\
\none & \none & 5 \\
\none & \none & 6 \\
\none & \none & 7 \\
\none & 2 & 8  \\
1 & 3
\end{ytableau}
\end{array}
\qquad
T_2 = \begin{array}{c}
\begin{ytableau}
\none & \none & \none & 4 \\
\none & \none & 5 & 10 \\
\none & \none & 6 \\
\none & \none & 7 \\
\none & \none & 8 \\
\none & 2 & 9  \\
1 & 3
\end{ytableau}
\end{array}
\qquad
T_3 = \begin{array}{c}
\begin{ytableau}
\none & \none & \none & 3 \\
\none & \none & 4 & 9 \\
\none & \none & 5 \\
\none & \none & 7 \\
\none & \none & 8 \\
\none & 1 & 10  \\
2 & 6
\end{ytableau}
\end{array}.
\end{displaymath}
Then 
\begin{displaymath}
\ytableausetup{mathmode, boxsize=1.2em}
\tab_{T_1} = \begin{array}{c}
\begin{ytableau}
3 & 2 \\
8 & 7 & 6  & 5 & 4 \\
1 \\
10 & 9
\end{ytableau}
\end{array}
\hskip 10mm
\tab_{T_2} = \begin{array}{c}
\begin{ytableau}
3 & 2 \\
9 & 8 & 7 & 6 & 5\\
1\\
10 & 4
\end{ytableau}
\end{array}
\hskip 10mm
\tab_{T_3} =
\begin{array}{c}
\begin{ytableau}
*(black!10)  6 & *(black!10) 1 \\
10 & 8 & 7 & 5 & 4 \\
2 \\
9 & *(black!10)  3
\end{ytableau}
\end{array}.
\end{displaymath}
Therefore $\Phi^\sigma_\alpha(T_i) = \tab_{T_i}$ for $i = 1,2$. 
However, $\Phi^\sigma_{\alpha}(T_3) = 0$ since $\tab_{T_3}$ does not satisfy the triple condition (see the shaded boxes).
\end{example}

\begin{lemma}\label{Lem:Surjective}
The map $\Phi^\sigma_\alpha$ is a surjective $\C$-linear map.
\end{lemma}
\begin{proof}
We begin with considering the ribbon diagram such that the size of the $i$th column is given by that of the $\sigma^{-1}(i)$th row of $\tcd(\alpha^{\rm c} \cdot \sigma)$. It is straightforward to check that this is the ribbon diagram of $\alpha$.

Let $\tau \in \SPCT^{\sigma}(\alpha^{\rm c} \cdot \sigma)$.
Define $T$ to be the filling of $\trd(\alpha)$ whose
$i$th column is filled with the entries in the $\sigma^{-1}(i)$th row of $\tab$ in increasing order for $i = 1, 2, \ldots, \ell(\alpha^{\rm c})$.
We claim that $T$ is a standard ribbon tableau of shape $\alpha$.
Note that the entries in each column of $T$ increase from top to bottom 
since those in each row of $\tau$ decrease from left to right.
To show that the entries in each row of $T$ increase from left to right,
we have to show that the entry in the uppermost box of the $i$th column of $T$
is less than that in the lowermost box of the $(i+1)$st column of $T$ for all $i$.
By the construction of $T$,
it is obvious since the entry in the lowermost box of the $i$th column of $T$
is less than that in the lowermost box of the $(i+1)$st column of $T$ for all $i$.

For our assertion, it suffices to show that $\Phi^\sigma_\alpha(T) = \tau$,
which is straightforward from the construction of $T$.
\end{proof}

Under $\Phi^\sigma_{\alpha}$,
the set $\{T \in \SRT(\alpha) \mid \Phi^\sigma_{\alpha}(T) \neq 0\}$ corresponds bijectively
to $\SPCT^\sigma_{\alpha^{\rm c} \cdot \sigma}$.
In particular, as seen in the proof of Lemma~\ref{Lem:Surjective}
the source tableau $T_0$ of $\bfP_\alpha$ maps to 
the canonical source tableau $\tauC$ of shape $\alpha^{\rm c} \cdot \sigma$ and type $\sigma$.

\begin{lemma}\label{Lem:0-Hecke homo}
The map ${\bf pr}_C \circ \Phi^\sigma_\alpha: \overline{\theta[\bfP_\alpha]} \to \bfS^\sigma_{\alpha^{\rm c} \cdot \sigma,C}$
is an $H_n(0)$-module homomorphism,
where ${\bf pr}_C: \bfS^\sigma_{\alpha^{\rm c} \cdot \sigma} \to \bfS^\sigma_{\alpha^{\rm c} \cdot \sigma,C}$ is the projection.
\end{lemma}
\begin{proof}
We simply write $\hPhi^\sigma_{\alpha}$ for ${\bf pr}_C \circ \Phi^\sigma_\alpha$.
To show our assertion it suffices to show that  
$$
\hPhi^\sigma_{\alpha}(\pi_i \star T) = \pi_i \cdot \hPhi^\sigma_{\alpha}(T)
$$ 
for all $i = 1, 2, \ldots, n-1$ and $T \in \SRT(\alpha)$.
For $i = 1, 2, \ldots, n$, 
let $c_i$ be the index of column of $T$ containing $i$.
Suppose that $i$ appears at the $r_i$th box from the bottom of the $c_i$th column of $T$.

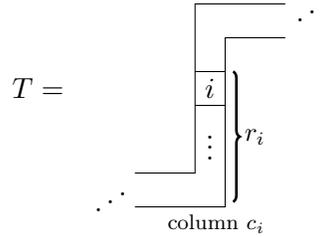
\begin{displaymath}
\begin{tikzpicture}
\def \hhh{4mm}
\def \vvv{4.5mm}
\draw (-\hhh*2,\vvv*0) -- (\hhh*1,\vvv*0) -- (\hhh*1,\vvv*5) -- (\hhh*3,\vvv*5);
\draw (-\hhh*2,\vvv*1) -- (\hhh*0,\vvv*1) -- (\hhh*0,\vvv*6) -- (\hhh*3,\vvv*6);
\draw (-\hhh*0,\vvv*3) -- (\hhh*1,\vvv*3);
\draw (-\hhh*0,\vvv*4) -- (\hhh*1,\vvv*4);
\node at (\hhh*0.5,\vvv*3.5) {\small $i$};
\node at (\hhh*0.7,-\vvv*0.5) {\tiny column $c_i$};
\node at (-\hhh*2.8,\vvv*0.4) {\reflectbox{$\ddots$}};
\node at (\hhh*3.9,\vvv*6) {\reflectbox{$\ddots$}};
\node at (\hhh*0.5,\vvv*2) {$\vdots$};
\node[left] at (-\hhh*4,\vvv*3.5) {\small $T =$};
\draw[decoration={brace,raise=3pt},decorate,line width=1pt] (\hhh*1.0,\vvv*4) -- (\hhh*1.0,\vvv*0.15)
node[right, midway,xshift = 3pt] {\footnotesize $r_i$};
\end{tikzpicture}
\end{displaymath}

{\it Case 1: $\pi_i \star T = T$.} 
If $\hPhi^\sigma_\alpha(T) = 0$, then there is nothing to prove. 
Suppose that $\hPhi^\sigma_\alpha(T) \neq 0$, that is, $\tau_T \in C$.
We claim that $i \notin \Des(\tau_T)$.

Observe that the condition $\pi_i \star T = T$ is equivalent to
\begin{equation*}
c_i \ge c_{i+1}.
\end{equation*}
If $c_{i} > c_{i+1}$, then the entry at $(\sigma^{-1}(c_i),1)$ is greater than the entry at $(\sigma^{-1}(c_{i+1}),1)$ in $\tab_T$.
Suppose that $i$ appears in the $k$th column of $\tau_T$.
If $i+1$ lies to weakly right of $i$ in $\tau_T$,
then the entry at $(\sigma^{-1}(c_i), k)$ is less than the entry at $(\sigma^{-1}(c_{i+1}), k)$.
This contradicts the assumption that $\tau_T \in C$.
Hence $i+1$ lies to strictly left of $i$ in $\tau_T$.
If $c_{i} = c_{i+1}$, then $i$ and $i+1$ appear in the same row in $\tab_T$. In both cases, $i \notin \Des(\tab_T)$, as required.

{\it Case 2: $\pi_i \star T = 0$.}
If $\hPhisa(T) = 0$, then there is nothing to prove.
Suppose that $\hPhisa(T) \neq 0$, that is, $\tau_T \in C$.
The condition $\pi_i \star T = 0$ implies that $i$ and $i+1$ are in the same row of $T$, thus $c_i + 1 = c_{i+1}$.
By construction, 
$i$ appears at the rightmost box in the $\sigma^{-1}(c_i)$th row 
and $i+1$ at the leftmost box in the $\sigma^{-1}(c_{i+1})$th row in $\tau_T$. 
We will show that $r_i=1$. 
Otherwise, the entry at $(\sigma^{-1}(c_i), 1)$ is greater than $i$.
Since $(\sigma^{-1}(c_{i+1}), 1)$ is filled with $i+1$,
we conclude that the entry of $(\sigma^{-1}(c_i), 1)$ is greater than that of $(\sigma^{-1}(c_{i+1}), 1)$.
This contradicts the assumption that $\tau_T \in C$.

As a consequence, $i$ and $i+1$ lie in the first column of $\tau_T$. 
This says that $i$ is an attacking descent of $\tab_T$, so $\pi_i \cdot \hPhisa(T) = 0$.
\smallskip

{\it Case 3: $\pi_i \star T = s_i \cdot T$.}
First, consider the case where $\hPhisa(T) = 0$.
Suppose that $\hPhisa(\pi_i \star T) \neq 0$, 
equivalently $\tau_{s_i \cdot T} \in C$.
For $j = i, i+1$,
let $c'_j$ be the index of column of $s_j \cdot T$ containing $j$.
Suppose that $j$ appears at the $r'_j$th box from the bottom of the $c'_j$th column of $s_j \cdot T$.
From the assumption $\pi_i \star T = s_i \cdot T$ 
we know that $i$ is strictly right of $i+1$ in $s_i \cdot T$,
that is, $c'_{i} > c'_{i+1}$.
Moreover, from the assumption $\tau_{s_i \cdot T} \in C$
we know that the entry at $(\sigma^{-1}(c'_{i}), k)$ is greater than that at $(\sigma^{-1}(c'_{i+1}), k)$ for all $k$ in $\tau_{s_i \cdot T}$,
which implies $r'_{i+1} < r'_{i}$.

\begin{enumerate}[label = {\rm (\roman*)}]
\item $\sigma^{-1}(c'_i) > \sigma^{-1}(c'_{i+1})$.
Since $i$ is in a lower row than  $i+1$ in $\tau_{s_i \cdot T}$,
from the triple condition we know that 
the column containing $i$ is not adjacent to that containing $i+1$ in $\tau_{s_i \cdot T}$, so $r'_{i+1}+1 < r'_{i}$.
Then $s_i \cdot \tau_{s_i \cdot T} \in C$.
But this is absurd since $s_i \cdot \tau_{s_i \cdot T}= \tau_T \notin C$.

\item $\sigma^{-1}(c'_i) < \sigma^{-1}(c'_{i+1})$. 
Since $i$ is in a higher row than  $i+1$ in $\tau_{s_i \cdot T}$ and $r'_{i+1} < r'_{i}$, we see that $s_i \cdot \tau_{s_i \cdot T} \in C$.
But this is absurd since $s_i \cdot \tau_{s_i \cdot T}= \tau_T \notin C$.
\end{enumerate}
From {\rm (i)} and ${\rm (ii)}$ we conclude that $\hPhisa(\pi_i \star T) = 0$.

Next, consider the case where $\hPhisa(T) \neq 0$.
From the assumption $\pi_i \star T = s_i \cdot T$ we know that $c_i < c_{i+1}$.
Since $\tau_{T} \in C$,  
the entry at $(\sigma^{-1}(c_{i}), k)$ is less than that at $(\sigma^{-1}(c_{i+1}), k)$ for all $k$ in $\tau_T$. 
Thus $r_{i}  \leq  r_{i+1}$ and $i$ is a descent of $\tau_T$.

\begin{enumerate}[label = {\rm (\roman*)}]
\item $\sigma^{-1}(c_i)<\sigma^{-1}(c_{i+1})$. 
Note that $i$ is in a higher row than  $i+1$ in $\tau_T$.
If $r_{i} = r_{i+1}$, then $i$ is an attacking descent of $\tau_T$, so $\pi_i \cdot \hPhisa(T) = 0$.
In addition, $\hPhisa(\pi_i \star T) = 0$ since $\tau_{s_i \cdot T} \notin C$.
If $r_{i}+1= r_{i+1}$, then $i$ is also an attacking descent of $\tau_T$, so $\pi_i \cdot \hPhisa(T) = 0$.
In addition, $\hPhisa(\pi_i \star T) = 0$ since $\tau_{s_i \cdot T}$ is not an SPCT.
If $r_{i} +1 < r_{i+1}$, then $i$ is a nonattacking descent of $\tau_T$, 
so $\pi_i \cdot \hPhisa(T) = s_i \cdot \hPhisa(T) = \hPhisa(s_i \cdot T) =
\hPhisa(\pi_i \star T)$.

\item $\sigma^{-1}(c_i)>\sigma^{-1}(c_{i+1})$.
Note that $i$ is in a lower row than $i+1$ in $\tau_T$.
If $r_{i} = r_{i+1}$, then $i$ is an attacking descent of $\tau_T$, so $\pi_i \cdot \hPhisa(T) = 0$.
In addition, $\hPhisa(\pi_i \star T) = 0$ since $\tau_{s_i \cdot T} \notin C$.
If $r_{i} < r_{i+1}$, then $i$ is a nonattacking descent of $\tau_T$,
so $\pi_i \cdot \hPhisa(T) = s_i \cdot \hPhisa(T) = \hPhisa(s_i \cdot T) =
\hPhisa(\pi_i \star T)$.
\end{enumerate}
From {\rm (i)} and ${\rm (ii)}$ we conclude that 
$\hPhisa(\pi_i \star T) = \pi_i \cdot \hPhisa(T)$.
\end{proof}

\begin{remark}\label{description of kernel}
An explicit description of the kernel of ${\bf pr}_C \circ \Phi^\sigma_\alpha$ can be given.
For a standard ribbon tableau $T$, let $T_{i, j}$ denote the entry at the box in $T$ located in the $j$th column and the $i$th row from the bottom of the $j$th column.
Define $\Omega_\alpha^\sigma$ to be the set of standard ribbon tableaux $T \in \SRT(\alpha)$ satisfying that either
\begin{enumerate}[label = {\bf K\arabic*.}]
  \item there exists a triple $(i,j,k)$ such that $i < j$ and $T_{k, i} > T_{k, j}$, or
  \item there exists a triple $(i,j,k)$ such that $i < j$, $\sigma^{-1}(i) < \sigma^{-1}(j)$, and $T_{k, i} > T_{k+1, j}$.
\end{enumerate}
Note that $T_{k,i} = (\tau_T)_{\sigma^{-1}(i), k}$.
Hence, for the standardization of the first column of $\tau_T$ to be $\sigma$,
$T_{1,i}$ should be less than $T_{1,j}$ whenever $i <j$.
Going further, 
for $\tau_T$ to be in $C$,
$T_{k,i}$ should be less than $T_{k,j}$ for all $k$ whenever $i <j$.
If this condition is satisfied, then
verifying that every triple in $\tau_T$ obeys the triple condition
can be done by showing that $T_{k,i} < T_{k+1,j}$ for all $k$ whenever $i<j$ and $\sigma^{-1}(i) < \sigma^{-1}(j)$.
From the above discussion it follows that $\ker({\bf pr}_C \circ \Phi^\sigma_\alpha)$ is equal to the $\mathbb{C}$-span of $\Omega_\alpha^\sigma$.
\end{remark}

We now state the main result in this section. 
\begin{theorem}\label{Thm:main Sec5}
Let $\alpha$ be a composition.
For any $\sigma \in \SG_{\ell(\alpha)}$, 
$\bfS^{\sigma}_{\alpha,C}$ is a homomorphic image of $\calP_{\set(\alpha \cdot \sigma^{-1})}$.
In particular, 
if $\alpha$ is a partition and $w_0$ the longest element in $\SG_{\ell(\alpha)}$,
then $\bfS^{w_0}_\alpha$ is a homomorphic image of $\calP_{\set(\alpha^\rev)}$.
\end{theorem}
\begin{proof}
Note that  
\begin{displaymath}
\calP_{\set(\alpha\cdot \sigma^{-1})} 
\underset{\text{\cite[Proposition 5.1]{16Huang}}}{\cong}  
\theta[\calP_{\set(\alpha\cdot \sigma^{-1})^{\rm c}}] 
\underset{\text{Theorem~\ref{Thm22}}}{\cong} 
\theta[\bfP_{(\alpha \cdot \sigma^{-1})^{\rm c}}]
\underset{ \text{\eqref{Eq:Isom1}}}{\cong}
\overline{\theta[\bfP_{(\alpha \cdot \sigma^{-1})^{\rm c}}]}.
\end{displaymath}
Now the first assertion follows from 
Lemma~\ref{Lem:Surjective} together with Lemma~\ref{Lem:0-Hecke homo}.
The second assertion follows from Corollary~\ref{Coro318}
and the equality $\alpha\cdot w_0 = \alpha^\rev$.
\end{proof}

\begin{corollary}\label{projective cover is projective}
Let $\alpha$ be a composition and $\sigma \in \SG_{\ell(\alpha)}$. Then 
$\bfS^{\sigma}_{\alpha,C}$ is projective if and only if 
either $\alpha = (1,\ldots,1)$ or 
$\alpha = (k,1,\ldots,1)$ for some $k \ge 2$ and $\sigma(1)=\ell(\alpha)$.
In this case, $\alpha$ is $\sigma$-simple and $\bfS^{\sigma}_{\alpha} \cong \calP_{\set(\alpha^\rmr)}$.
\end{corollary}

\begin{proof}
In case where $\alpha = (1,1,\ldots,1)$, the assertion clearly holds since $|\SPCTsa| = 1 = |\SRT((\alpha \cdot \sigma^{-1})^{\rm c})|$.
In the following, we assume that $\alpha \neq (1,1,\ldots,1)$. 

For a composition $\beta$, we claim that $\ker({\bf pr}_C \circ \Phi^\sigma_\beta) = 0$ if and only if $\sigma^{-1}(\ell(\beta))=1$ and $\beta$ is a partition of hook shape. 
The ``if'' part is obvious from the definition of ${\bf pr}_C \circ \Phi^\sigma_\beta$.
To prove the ``only if'' part,
we define $T_\beta \in \SRT(\beta)$ to be the standard ribbon tableau obtained by filling $\trd(\beta)$ with the entries $1,2, \ldots, |\beta|$ from left to right and from top to bottom.
Since ${\bf pr}_C \circ \Phi^\sigma_\beta(T_\beta)$ is nonzero,
$(T_\beta)_{1,1}$ is less than $(T_\beta)_{1, \ell}$ where $\ell$ is the index of the rightmost column of $T_\beta$.
Hence $\beta$ is a partition of hook shape.
If $\sigma^{-1}(\ell(\beta)) \neq 1$, then
$\tau_{T_\beta}$ is not an SPCT since the entry at $(1, 1)$ is greater than the entry at $(\sigma^{-1}(\ell(\beta)), 2)$ and $(1,2) \not\in \tcd(\beta^c \cdot \sigma)$.
This verifies our claim.

Note that ${\bfS}^\sigma_{\alpha,C}$ is projective if and only if $\ker({\bf pr}_C \circ \Phi^\sigma_\beta) = 0$ with $\beta = (\alpha \cdot \sigma^{-1})^{\rm c}$. Thus the first assertion follows from the claim.
If $\alpha = (k, 1, \ldots, 1)$ for some $k$ and $\sigma^{-1}(\ell(\alpha))=1$, then 
$\alpha$ is clearly $\sigma$-simple.
Hence the second assertion follows from Theorem~\ref{Thm:main Sec5} since
\[
{\bfS}^\sigma_{\alpha} = {\bfS}^\sigma_{\alpha, C}  \cong  
\overline{\theta[\bfP_{(\alpha \cdot \sigma^{-1})^{\rm c}}]} \cong
\calP_{\set(\alpha\cdot \sigma^{-1})}  = \calP_{\set(\alpha^\rmr)} \cong \bfS^{w_0}_{\alpha}.\qedhere
\]
\end{proof}

\begin{remark}\hfill
\begin{enumerate}[label = {\rm(\alph*)}]
\item 
Let $R$ be a left artin ring and $P$ a finitely generated indecomposable projective $R$-module. 
It is known that $P/M$ is an indecomposable $R$-module for each submodule $M$ of $P$ with $M \neq P$ (for example see~\cite{95ARS}).
Therefore it follows directly from Theorem~\ref{Thm:main Sec5} that 
$\bfS^{\sigma}_{\alpha,C}$ is an indecomposable $H_n(0)$-module.
\item  
Theorem~\ref{Thm:main Sec5} is not valid for a non-canonical class in general. For instance, let $E$ be the class generated by the source tableau
\begin{displaymath}
\ytableausetup{mathmode, boxsize=1.2em}
\tau_0 = 
\begin{array}{c}
\begin{ytableau}
4 & 3 \\
5 & 2 \\
1 
\end{ytableau}
\end{array}.
\end{displaymath}
If the projective cover of $\bfS^{231}_{(2,2,1),E}$ is a PIM, then $\TOP(\bfS^{231}_{(2,2,1),E})$ is isomorphic to $\SOC(\theta[\bfP_{(2,3)}]) \cong \TOP(\theta[\bfP_{(3,2)}])$. 
Here $\SOC(\theta[\bfP_{(2,3)}])$ denotes the \emph{socle} of $\theta[\bfP_{(2,3)}]$, the sum of all simple submodules of $\theta[\bfP_{(2,3)}]$.
However, one can easily see that there are no surjective $H_5(0)$-module homomorphisms from $\theta[\bfP_{(3,2)}]$ to $\bfS^{231}_{(2,2,1),E}$.
\end{enumerate}
\end{remark}

\appendix 
\counterwithin{theorem}{section}
\section{The roadmap to the proof of Theorem~\ref{thm: eq class indecomp}}
\label{sec: indecomp}
The proof of~\cite[Theorem 4.11]{19Konig}, which amounts to Theorem~\ref{thm: eq class indecomp} when $\sigma = \id$, is organized in a very complicated way and depend on many preliminary arguments. For readers' understanding, we give the roadmap to the proof of~\cite[Theorem 4.11]{19Konig} (see {\sc Figure}~\ref{Fig:roadmapA}).

\begin{figure}[h]
\begin{displaymath}
\begin{tikzpicture}[node distance=.6cm, start chain=going below]
\node[punktchain] (A2) 
{\cite[Theorem 4.11]{19Konig}}
child[<-,thick]{node[small punktchain, join] (B2)      {\cite[Lemma 4.9]{19Konig}}
child{node[small punktchain, join] (C2) {\cite[Lemma 4.3]{19Konig}}
child{node[punktchain, join] (D2) {\cite[Theorem 2.19]{19Konig}}}}}
child[<-,thick]{node[punktchain, left=1.5cm of B2] (B1) {\cite[Theorem 2.23]{19Konig}}}
child[<-,thick]{node[punktchain, right=1.2cm of B2] (B3) {\cite[Corollary 2.24]{19Konig}}
child{node[small punktchain, join] (C3) 
{\cite[Lemma 2.2]{19Konig}}}};
\draw[|-,-|,->,thick]  (B1.south) |-+(0,-.35cm)-| (B3.south);
\node[small punktchain, right=1.25cm of A2,<-,thick] (A3)  
{\cite[Lemma 4.1]{19Konig}};
\draw[->,thick]  (A3.west)--(A2.east);
\node[small punktchain, left=2cm of C2,<-,thick,fill=black!10] (C1) {\cite[Lemma 4.7]{19Konig}}
child{node[punktchain, join,fill=black!10] (D1) {\cite[Proposition 3.8]{19Konig}}};
\draw[<-,thick]  ([xshift=0]C1.east) -- (C2.west);
\node[small punktchain] (E1)   {\cite[Lemma 4.5]{19Konig}}
child{node[small punktchain, join] (F1)   {\cite[Lemma 4.4]{19Konig}}};
\draw[->,thick]  (F1.west)-- ([xshift=-25]F1.west) -- ([xshift=-25]C1.west) -- (C1.west);
\draw[->,thick]  ([yshift=5]C1.east)--([xshift=20,yshift=5]C1.east)--
([xshift=-37,yshift=3]B2.west)--([yshift=3]B2.west);
\draw[->,thick]  (E1.east) -- ([xshift=33]E1.east)
-- ([xshift=-24,yshift=-5]B2.west) -- ([yshift=-5]B2.west);
\draw[->,thick]  ([yshift=5]E1.east) -- ([xshift=20,yshift=5]E1.east)
-- ([xshift=20,yshift=-5]C1.east) -- ([yshift=-5]C1.east);
\node[punktchain, right=0.7cm of D2,<-,thick] (D3) {\cite[Proposition 2.25]{19Konig}}
child{node[small punktchain, join] (E3) 
{\cite[Lemma 2.21]{19Konig}}};
\draw[->,thick]  (D3.west)--([yshift=-12]B2.east);
\end{tikzpicture}
\end{displaymath}
\caption{The roadmap to the proof of~\cite[Theorem 4.11]{19Konig}}
\label{Fig:roadmapA}
\end{figure}

For the proof of Theorem~\ref{thm: eq class indecomp}, we need to extend these arguments to SPCTs.
After careful examination,
we see that this can be done successfully without difficulty except for \cite[Proposition 3.8]{19Konig} and \cite[Lemma 4.7]{19Konig}.
Since \cite[Proposition 3.8]{19Konig} does not work for SPCTs, we design Lemma~\ref{Lem:App1} as an alternative. Using this, we make an SPCT version Lemma~\ref{lem: modifeid 4.7} of~\cite[Lemma 4.7]{19Konig}.
With this modification, we can prove Theorem~\ref{thm: eq class indecomp} following the roadmap in {\sc Figure}~\ref{Fig:roadmapA}.

The rest of this section is devoted to the verification of Lemma~\ref{Lem:App1} and Lemma~\ref{lem: modifeid 4.7}.
The lemmas necessary for the proof will be stated without proof since they can be extended to SPCTs straightforwardly.
From now on, we assume that $\alpha \models n$, $\sigma \in \SG_{\ell(\alpha)}$, and $E \in \mathcal{E}^\sigma(\alpha)$.

\begin{lemma}\label{Lemma32}
{\rm (cf.~\cite[Lemma 4.3]{15TW})} 
Let $\tau, \tau' \in E$ be such that $\pi_{j_1}  \cdots \pi_{j_r}\cdot \tau = \tau'$ for some $j_1, \ldots, j_r \in [n-1]$. 
If $s_{i_1} \cdots s_{i_p}$ is a reduced expression for ${\rm col}_{\tau'} ({\rm col}_{\tau})^{-1}$, then $\pi_{i_1} \cdots \pi_{i_p} \cdot \tau = \tau'$. 
\end{lemma}

Define a partial order $\preccurlyeq$ on $\SPCTsa$ by
$$
\tau \preccurlyeq \tau' 
\quad \text{if and only if} \quad 
\text{$\tau' = \pi_\gamma \cdot \tau$ for some $\gamma\in\SG_n$.}
$$
\begin{lemma}{\rm (cf.~\cite[Theorem 2.23]{19Konig})}
\label{Prop:Thm618}
Given $E \in \mathcal{E}^\sigma(\alpha)$, let $\tau_0$ and $\tau_1$ be the source tableau and the sink tableau in $E$, respectively.
Then the poset $(E,  \preccurlyeq )$ has the structure of a graded lattice and is isomorphic to the subinterval
$[{\rm col}_{\tau_0}, {\rm col}_{\tau_1}]$ under the weak Bruhat order $\leq$ on $\SG_n$.
\end{lemma}

\begin{lemma}
\label{Lem:App1}
For any $\tau \in E$ with $\tau \neq \tau_0$, set 
\begin{equation}\label{Def:i}
\rmd(\tau) :=  \max\left\{ 1 \leq k \leq n-1   \mid  \tau_0^{-1}(k) \neq \tau^{-1}(k) \right\} \quad \text{and} \quad \rho_\tau := {\rm col}_{\tau}  ({\rm col}_{\tau_0})^{-1}.
\end{equation}
If $s_{i_p} \cdots s_{i_1}$ is a reduced expression for $\rho_\tau$,
then $i_q < \rmd(\tau)$ for $q = 1,2,\ldots, p$.
\end{lemma}
\begin{proof}
We use induction on $\ell(\rho_\tau)$.
From Lemma~\ref{Prop:Thm618} it follows that
the poset $(E, \preccurlyeq )$ has a finite rank $f$.
If $\ell(\rho_\tau) = 1$, that is, $\rho_\tau = \pi_i$ for some $1 \le i \le n-1$, 
then the assertion is true since $\rmd(\tau) = i+1$.
Suppose that the assertion is true for all $\tab \in E$ with $\ell(\rho_\tau) = p$ for $1 \leq p < f$.
Let $\tab' \in E$ with $\ell(\rho_{\tau'}) = p+1$ and $s_{i_{p+1}} \cdots s_{i_1}$ be a reduced expression of $\rho_{\tab'}$.
By Lemma~\ref{Prop:Thm618}, we see that $\tab' = \pi_{i_{p+1}} \cdot \tab$ and $s_{i_p} \cdots s_{i_1}$ is a reduced expression of $\rho_{\tab}$. 
If $\rmd(\tau') = \rmd(\tau)$, then $i_{p+1} < \rmd(\tau)$.
Otherwise, by the definition~\eqref{Def:i}, we see that $\rmd(\tau') > \rmd(\tau)$.
Combining this with the equality $\{k \mid \tau^{-1}(k) \neq \tau'^{-1}(k)\} = \{ i_{p+1}, i_{p+1}+1\}$ yields
that $\rmd(\tau') = i_{p+1}+1$.
Now the assertion follows from the induction hypothesis.
\end{proof}

\begin{lemma}{\rm (cf.~\cite[Lemma 4.3]{19Konig})}
\label{Lem:Lemma43}
For any $\tau \in E \setminus \{\tau_0\}$,
if $\Des(\tau) \subseteq \Des(\tau_0)$, then $\rmd(\tau) \in \Des(\tau_0)$.
\end{lemma}

\begin{lemma}{\rm (cf.~\cite[Lemma 4.5]{19Konig})}
\label{Lem:Lemma45}
For $i \in \Des(\tau_0)$, set
\begin{equation*}
\mathrm{a}(i) :=  \min \left\{ k > i  \mid i \text{ and } k \text{ are attacking in } \tau_0 \right\}.
\end{equation*}
\begin{enumerate}[label = {\rm (\alph*)}]
\item There exists at least one $k > i$ such that $i$ and $k$ are attacking in $\tau_0$. 
 
\item For all $i+1 \le k \le j-1$, $i$ and $k$ are nonattacking and $i$ is strictly left of $k$ in $\tau_0$.
\end{enumerate}
\end{lemma}
For any $\tau \in E \setminus \{\tau_0\}$ with $\Des(\tab) \subseteq \Des(\tab_0)$, set $\mathrm{a}(\tau) := \mathrm{a}(\mathrm{d}(\tau))$. This is well-defined by Lemma~\ref{Lem:Lemma43} and Lemma~\ref{Lem:Lemma45}.

\begin{lemma}{\rm (cf.~\cite[Lemma 4.7]{19Konig})}\label{lem: modifeid 4.7}
Let $\tau \in E \setminus \{\tau_0\}$ with $\Des(\tab) \subseteq \Des(\tab_0)$.
For each $\rmd(\tau)+1 \le k \le \rma(\tau)$, 
$\rmd(\tau)$ and $k$ are nonattacking and $\rmd(\tau)$ is strictly left of $k$ in $\tau$.
\end{lemma}
\begin{proof}
By Lemma~\ref{Lemma32}, we can write $\tau$ as $\pi_{i_p} \cdots \pi_{i_1} \cdot \tau_0$, where $s_{i_p} \cdots s_{i_1}$ is a reduced expression of ${\rm col}_\tau {\rm col}^{-1}_{\tau_0}$.
The proof can be done essentially in the same way as in~\cite[Lemma 4.7]{19Konig}.
The only difference lies in the manner to derive the inequality 
$$
i_q \leq \rmd(\tau)-1 \qquad \text{for } 1 \leq q \leq p.
$$
It follows from Lemma~\ref{Lem:App1}, whereas it follows from~\cite[Proposition 3.8]{19Konig} in~\cite{19Konig}.
\end{proof}

\noindent {\bf Acknowledgments.}
The authors are grateful to the anonymous referees for their careful readings of the manuscript and valuable advice.

\bibliographystyle{abbrv}
\bibliography{references}
	
\end{document}